\documentclass[11pt]{amsart}

\usepackage{amsmath,amssymb}

\usepackage{geometry}
\usepackage{amscd}
\usepackage{amsfonts, amssymb} 
\usepackage{stmaryrd}
\usepackage[all]{xy}
\usepackage{latexsym}
\usepackage{cite}
\usepackage[usenames]{color}
\usepackage{xcolor}
\usepackage{url}
\usepackage{boxedminipage}
\usepackage{amsbsy}
\usepackage{framed}
\usepackage[colorlinks=false, linktocpage=true]{hyperref}

\usepackage{enumitem}

\usepackage[bbgreekl]{mathbbol}

\usepackage{tikz-cd}

\usepackage{datetime2}
\newcommand{\frs}{\mathfrak s}

\newcommand{\fm}{\mathfrak m}

\newcommand{\fM}{\mathfrak M}

\newcommand{\cB}{{\mathcal B}}
\newcommand{\cI}{{\mathcal I}}
\newcommand{\cC}{{\mathcal C}}

\newcommand{\cE}{{\mathcal E}}

\newcommand{\cW}{{\mathcal W}}
\newcommand{\cP}{{\mathcal P}}

\newcommand{\cX}{{\mathcal X}}
\newcommand{\R}{\mathbb{R}}

\newcommand{\bGamma}{ \mathbb{\Gamma}}

\usepackage{mathrsfs}

\newcommand{\scA}{\mathscr{A}}
\newcommand{\scB}{\mathscr{B}}
\newcommand{\scC}{\mathscr{C}}
\newcommand{\scD}{\mathscr{D}}
\newcommand{\scE}{\mathscr{E}}
\newcommand{\scG}{\mathscr{G}}
\newcommand{\scI}{\mathscr{I}}
\newcommand{\scJ}{\mathscr{J}}
\newcommand{\scK}{\mathscr{K}}

\newcommand{\scM}{\mathscr{M}}
\newcommand{\scO}{\mathscr{O}}
 
\newcommand{\scS}{\mathscr{S}}
\newcommand{\scT}{\mathscr{T}}

\newcommand{\Euc}{\mathsf{Euc}}
\newcommand{\Set}{\mathsf{Set}}
\newcommand{\Aff}{\mathsf{\cin Aff}}
\newcommand{\cring}{{\cin\mathsf{Ring}}}
\newcommand{\Man}{\mathsf{Man}}
\newcommand{\lcrs}{\mathsf{L\cin RS}}
\newcommand{\LCRS}{\mathsf{L\cin RS}}

\newcommand{\Open}{\mathsf{Open}}

\DeclareMathAlphabet{\mathpzc}{OT1}{pzc}{m}{it}

\newcommand{\inv}{^{-1}}

\newcommand{\op}[1]{{#1}^{\mbox{\sf{\tiny{op}}}}}
\newcommand{\comp}[1]{{#1}_{\mbox{\sf{\tiny{comp} }}}}
\newcommand{\cin}{C^\infty}

\newcommand{\uu}[1]{\underline{#1}}
\newcommand{\ti}[1]{\tilde{#1}}

\DeclareMathOperator{\CDer}{\cin Der}
\DeclareMathOperator{\Hom}{Hom}
\DeclareMathOperator{\id}{id}

\DeclareMathOperator{\Spec}{Spec}
\DeclareMathOperator{\ev}{\mathsf{ev}}
\DeclareMathOperator{\pr}{\mathsf{pr}}

\numberwithin{equation}{section}

\theoremstyle{definition}
\newtheorem{thm}{Theorem}[section]
\newtheorem{lemma}[thm]{Lemma}
\newtheorem{theorem}[thm]{Theorem}

\newtheorem{corollary}[thm]{Corollary}
\newtheorem*{corollary*}{Corollary}
\newtheorem*{claim*}{Claim}

\newtheorem{definition}[thm]{Definition}
\newtheorem{remark}[thm]{Remark}
\newtheorem{example}[thm]{Example}
\newtheorem{notation}[thm]{Notation}
\newtheorem {construction}[thm]{Construction}

\begin{document}
\title{ Vector fields and flows on  $\cin$-schemes.} 
\author{ Eugene Lerman}
\setcounter{tocdepth}{1}

\begin{abstract} We prove that a vector field on an affine
  $\cin$-scheme $\Spec(\scA)$ has a flow provided the $\cin$-ring
  $\scA$ is finitely generated.  If the vector field is complete then
  the flow is the target map of a groupoid internal to the category of
  $\cin$-schemes.
  \end{abstract}
\maketitle
\tableofcontents

\section{Introduction}

The main goal of this paper is to prove that vector fields on
$\cin$-schemes have (unique maximal) integral curves and that these
integral curves can be put together to form  flows.

Singular spaces arise naturally in differential geometry and its
applications in physics and engineering.  The singularities may come
from constraints or from taking quotients by symmetries or from the
combination of the two (as in singular Marsden-Weinstein-Meyer
reduction).  Over the last seventy plus years many approaches to
singular spaces in differential geometry have been developed ranging
from the diffeological spaces of Souriau and Chen (see \cite{IZ}), the differential
spaces in the sense of Sikorski \cite{Si}, differential spaces in the sense
of Spallek \cite{Spa} and $\cin$-schemes of Dubuc \cite{Dubuc} to derived differential
geometry in its various variants (which started with \cite{Spi}).

Recently I started developing tools of differential geometry for
$\cin$-ringed spaces: differential forms \cite{L-forms}, Cartan
calculus \cite{L-C} and, jointly with Karshon, flows of vector fields
on differential spaces in the sense of Sikorski \cite{KL}.  Sikorski
differential spaces are fairly elementary: for one thing their
structure sheaves are sheaves of {\em  functions}. On the other hand 
differential spaces are somewhat limited, and more general local
$\cin$-ringed spaces appear naturally in differential geometry \cite{Joy}.

One motivation for constructing flows of vector fields on $\cin$-schemes
comes from Poisson geometry. In finite dimensional (differential)
Poisson geometry there are Poisson algebras that are not algebras of
functions on any space.  Such algebras, for instance, arise in the
Sniatycki-Weinstein algebraic reduction \cite{SW} (see \cite{AGJ} a
detailed comparison of algebraic and geometric reductions in
symplectic geometry).  It is straightforward to define a $\cin$-ring
analogue of a Poisson algebra; all algebra of smooth functions on
Poisson manifolds carry this structure.  Many other Poisson algebras
of interest in differential geometry also carry this structure.  One
can show that Dubuc's spectrum functor $\Spec$ \cite{Dubuc} sends a
Poisson $\cin$-ring/algebra to a Poisson $\cin$-scheme.  Sections of
the structure sheaf of a Poisson $\cin$-scheme then give rise to
Hamiltonian vector fields \cite{L_PS}.  It is natural to want to ensure that these
Hamiltonian vector fields have flows.  This is what the main result of
the paper, Theorem~\ref{thm:main} does.  $\cin$-Poisson schemes will
be discussed elsewhere \cite{L_PS}.\footnote{ We note parenthetically that in the
  setting of algebraic geometry the construction of 
Poisson schemes from  Poisson algebras is well-known.  This
  is mentioned, for example, by Polishchuk \cite{Po}.  There is a
 construction in \cite{Ka}.  However, given a Poisson manifold
  $(M, \pi)$ the spectrum of the Poisson algebra $\cin(M)$ in the
  usual algebra-geometric sense is not $(M, \cin_M)$.}

The present paper is meant to be accessible to differential geometers
who are not fluent in $\cin$-rings.  For this reason we review the
definition of $\cin$-rings and some related notions in the Appendix~\ref{app}.  A reader unfamiliar with $\cin$-rings may wish to consult the appendix while reading the rest of the introduction.\\

To state the main result of the paper more precisely  we need to
define integral curves of vector fields on ringed spaces and to define
flows of vector fields.  We do this presently  using the terms defined elsewhere in the paper. Then, after stating
the main result,  we describe the organization of the paper.

\begin{notation} \label{notation:presheaf_maps} We denote the
  collection of all open subsets of a space $X$ by $\Open(X)$.
  As usual a   map $f:\scS\to \scT$ of presheaves over a
  space $X$ is a collection of maps $\{f_U: \scS(U) \to
  \scT(U)\}_{U\in \Open(X) }$, so that for all pairs $W\subset U$, $f_U|_W = f_W$.
\end{notation}  
\begin{definition} \label{def:vf}
A {\sf vector field} on a local $\cin$-ringed space $(X, \scO_X)$ is a
map of sheaves {\em of real vector spaces} $v= \{v_U\}_{U\subset X_A}:\scO_X\to \scO_X$ so that for any
open subset $U\subseteq X$ the corresponding  map $v_U:\scO_X(U)\to \scO_X(U)$ is a
$\cin$-derivation. 
\end{definition}

\begin{notation} \label{not:vf}
We denote the set of all vector fields on a local $\cin$-ringed space
$(X, \scO_X)$ by $\cX(X,\scO_X)$.  
\end{notation}

A standard way of defining an integral curve of a vector field $v$ on a
manifold $M$ with initial condition $p$ is to define it as a smooth
map $\gamma: I \to M$, where $I\subset \R$ is an open interval
containing 0, subject to two conditions:
\begin{enumerate}
\item $\gamma (0)= p$ and
\item $\left. \frac{ d}{dt} \right|_t  (f\circ \gamma) = X(f) (\gamma
  (t))$ for all $t\in I$ and all smooth functions $f\in \cin(M)$.
\end{enumerate}
Therefore an integral curve of $v$ is a smooth map $\gamma:I \to M$
that relates the vector fields $\frac{ d}{dt}$ on $I$ and $v$ on $M$.
This definition generalizes to vector fields on local $\cin$-ringed
spaces.   We first define related (semi-conjugate) vector fields on local $\cin$-ringed spaces.

\begin{definition} \label{def:1.4}
Let $\uu{f}= (f, f_\#): (X,\scO_X)\to (Y,\scO_Y)$ be a map of
  local $\cin$-ringed spaces. The vector fields ${v}\in \cX (X,\scO_X)$
  and ${w}\in \cX(Y, \scO_Y)$ are {\sf $\uu{f}$-related} if the diagram
\[
  \xy
(-10,6)*+{\scO_Y }="1";
(14,6)*+{f_* \scO_X}="2";
(-10, -6)*+{\scO_Y}="3";
 (14, -6)*+{f_* \scO_X}="4";
{\ar@{->}_{w} "1";"3"};
{\ar@{->}^{f_*v} "2";"4"};
{\ar@{->}^{f_\#} "1";"2"};
{\ar@{->}_{f_\#} "3";"4"};
\endxy
\] commutes.  That is, for every open set $U\subset Y$ the derivations
$w_U\in \CDer(\scO_Y(U)) $ and $v_{f\inv(U)} \in \CDer (\scO_X (f\inv
(U))$ are $f_{\#,U} $-related (Notation~\ref{not:der} and Defintion~\ref{def:related-der}).
\end{definition}

The domains of definition of integral curves of vector fields on
manifolds (without boundary) are usually defined as open intervals.
This does not work  for manifolds with boundary as the following
simple example demonstrates.  Consider the vector field
$\frac{\partial}{\partial x}$ on the 
closed disk $M:= \{(x,y)\in \R^2 \mid x^2 + y^2 \leq 1\}$.  Its
maximal integral curves are closed intervals with the exception of the
curves through $(0,1)$ and $(0, -1)$, which are only defined for time
$t=0$.  See \cite{KL} for a related discussion of vector fields and flows on
arbitrary subsets of $\R^n$ and, more generally, on subcartesian
spaces.
\begin{definition}
An {\sf interval} (containing 0) is a connected subset $I$ of the real
line $\R$ (containing 0).
\end{definition}
\begin{remark}
Thus an interval can be open, closed, half-open, bounded, bounded
below or above, unbounded or a single point.   If an interval $I$ is a
singleton we view it as a 0-dimensional manifold.  Otherwise it's a
1-dimensional manifold (which may or may not have a boundary).
\end{remark}

\begin{remark}
Unless the interval $I$ is a point, 
$\frac{d}{dt}: \cin_I \to \cin_I$ is a vector field on a local
$\cin$-ringed space $(I, \cin_I)$.
\end{remark}

\begin{definition} \label{def:int_curve}
An {\sf integral curve} $\uu{\sigma}_p$ of a vector field ${v}$ on a local $\cin$-ringed
space $(X, \scO_X)$ with initial
condition $p\in X$ is
\begin{enumerate}
\item a map  $\uu{\sigma}_p = (\sigma_p, (\sigma_p)_\#): (\{0\},
  \cin_{\{0\}})\to (X, \scO_X)$ with $\sigma_p(0) = p$  {\em  or}
\item a  map $\uu{\sigma}_p = (\sigma_p, (\sigma_p)_\#): (I, \cin_I)\to (X, \scO_X)$ of local $\cin$-ringed spaces from
  a 1-dimensional interval $I$ containing 0 to $ (X, \scO_X)$ so that
\begin{enumerate}
\item $\frac{d}{dt}$ and $v$ are $\uu{\sigma}$-related and
\item $\sigma_p(0) ={p}$.  
\end{enumerate}    
\end{enumerate}   
\end{definition}

\begin{notation} \label{notation:2.15}
Let $v$ be a vector field on a local $\cin$-ringed space $(X, \scO_X)$
and $p\in X$ a point.  We denote  the maximal integral curve of $v$ with
initial condition $p$ (if it exists) by $\uu{\mu}_p
=(\mu_p,  (\mu_p)_\#)$ and
its interval of definition by $K_p$.  
\end{notation}

\begin{definition}\label{def:flow} We use Notation~\ref{notation:2.15} above.  A {\sf
    flow} of a vector field ${v}$ on a local $\cin$-ringed space
  $(X, \scO_X)$ is a map $\uu{\Phi}: (\cW, \scO_\cW)\to (X, \scO_X)$ of
  local $\cin$-ringed spaces so that
\begin{enumerate}
\item $\cW = \bigcup_{p\in X} \{p\}\times K_p \quad (\subseteq X\times
  \R)$  as sets (where $K_p$ is as in  Notation~\ref{notation:2.15}),
\item for any point $p\in X$ there is a map
  $\uu{\zeta}_p = (\zeta_p, (\zeta_p)_\#): (K_p, \cin_{K_p}) \to (\cW,
  \scO_\cW)$ with $\zeta_p (t) = (p, t)$ for all $t\in K_p$ and
\item    for all points
$p \in X$ the composite  $\uu{\Phi}\circ \uu{\zeta}_p$ is the maximal
integral curve $\uu{\mu}_p$ of ${v}$.
\end{enumerate}  
\end{definition}  

The main result of the paper is the following theorem: %
\begin{theorem}
  \label{thm:main}
  Let $\scA$ be a finitely generated and germ determined $\cin$-ring,
  $(X_\scA, \scO_\scA) = \Spec(\scA)$ the corresponding affine
  $\cin$-scheme.  Then any vector field ${v}$ on $\Spec(\scA)$ has
  flow.

  In particular for any point $p\in X_\scA$ there exists a
  unique maximal integral cuve
  $\uu{\mu}_p = (\mu_p, (\mu_p)_\#)$ of the vector field ${v}$ with initial
  condition $\mu_p(0) = p$.
\end{theorem}

\subsection*{Organization of the paper} In the next section, Section~\ref{sec:2},
we recall the notion  of localization of  $\cin$-rings, local $\cin$-ringed
spaces and their morphisms and, following Joyce \cite{Joy}, review a
construction of Dubuc's spectrum functor $\Spec$ from $\cin$-rings to
local $\cin$-ringed spaces.   In Section~\ref{sec:3} we show that
$\Spec$ sends derivations of $\cin$-rings to vector fields on the
corresponding affine scheme.  Moreover $\cin$-derivations related by
morphisms of $\cin$-rings get sent by $\Spec$ to 
vector fields related by maps of local $\cin$-ringed spaces.
In Section~\ref{sec:int} we show that vector fields on finitely
generated affine schemes have unique maximal integral curves.   In Section~\ref{sec:flows} we put
maximal integral curves together to obtain flows. In
Section~\ref{sec:6}  we show that for a complete  vector field the
  the flow is the target map of a groupoid internal to the category of
  $\cin$-schemes.

\subsection*{Acknowledgements} I am grateful to Yael Karshon for many
fruitful discussions about differential spaces, $\cin$-rings and
related mathematics.  I would also like to thank Rui Loja Fernandes
for many conversations.  In particular the proof of
Lemma~\ref{lem:5.1} is essentially due to him.

\section{Preliminaries} \label{sec:2}
\begin{remark}
Unless noted otherwise all manifolds  are smooth, second countable and Hausdorff.
\end{remark}

We start by recalling the notion of a
$\cin$-ringed space and of a local $\cin$-ringed space.

\begin{definition}
A {\sf $\cin$-ringed space} is a pair $(X, \scO_X)$ where $X$ is a
topological space and $\scO_X$ is a sheaf of $\cin$-rings on $X$.

A {\sf local} $\cin$-ringed space is a $\cin$-ringed space so that all
the stalks $\scO_{X,x}$ of the sheaf $\scO_X$ are local $\cin$-rings
in the sense of Definition~\ref{def:loc_ring}.
\end{definition}

\begin{definition}
A {\sf morphism} (or a {\sf map}) of local $\cin$-ringed spaces from $(X,\scO_X)$ to
$(Y,\scO_Y)$ is a pair $\uu{f}= (f, f_\#)$ where $f: X\to Y$ is continuous and
$f_\#: \scO_Y\to f_* \scO_X$ is a map of sheaves of $\cin$-rings on the space $Y$.
\end{definition}
\begin{remark}
  Note that the map $(f,f_\#): (X,\scO_X) \to (Y,\scO_Y)$ of local
  $\cin$-ringed spaces induces, for every point $x\in X$,  a map of
  stalks $f_x: \scO_{Y,f(x)} \to \scO_{X, x}$
  
  In algebraic geometry one defines a morphism of locally ringed
  spaces to be a pair $(f,f_\#): (X,\scO_X) \to (Y,\scO_Y)$ so that
  the induced maps on stalks $f_x$ are all local: they are required to
  take the unique maximal ideal to the unique maximal ideal.     For
  local $\cin$-ringed spaces this is unnecessary because any map of
  local $\cin$-rings automatically preserves maximal ideals (Lemma~\ref{lem:2.20}).
\end{remark}

\begin{definition}[The category $\LCRS$] \label{def:lcrs}
Local $\cin$-ringed spaces and their morphisms form a category which
we denote by $\LCRS$.
\end{definition}  

\begin{remark}
There is a fully faithful functor from the category $\Man$ of
$\cin$-manifolds to local $\cin$-ringed spaces.  It sends a manifold
$M$ to the pair $(M, \cin_M)$ where $\cin_M$ is the sheaf of smooth
functions on $M$.   It sends a smooth map $f:M\to N$ to $\uu{f} = (f,
f_\#): (M, \cin_M) \to (N, \cin_N)$ where $f_\#: \cin_N\to f_* \cin_M$
is $f^*$, the pull-back by $f$.  See \cite[Corollary~4.27]{Joy}.
\end{remark}

\subsection{The $\cin$-ring spectrum functor $\Spec$}

There is a (contravariant) functor $\Gamma$, the global sections
functor, from the category of local
$\cin$-ringed spaces $\LCRS$ %
to the (opposite) category  $\cring$ of
$\cin$-rings.  Namely, given a
$\cin$-ringed space $(M, \scO_M)$, the functor $\Gamma$ sends it to the ring
$\scO_M(M)$ of global sections.   Given a map $(f,f_\#):
(N, \scO_N)\to  (M,\scO_M) $ of local $\cin$-ringed spaces we have a map of $\cin$-rings
$f_\#:\scO_M(M) = f_*(\scO_N)(M) = \scO_N(f\inv (M) ) = \scO_N(N)$.  In other words
\[
\Gamma: \LCRS \to \op{\cring}
\] 
is given by
\[
\Gamma( (N, \scO_N)) \xrightarrow{(f,f_\#)} (M,\scO_M) ) :=
\scO_M(M)\xrightarrow{f_\#} \scO_N (N).
\]  
It is easy to see that $\Gamma$ is in fact a functor.  Thanks to a
theorem of Dubuc \cite{Dubuc}, the global section functor $\Gamma$ has
a right adjoint $\Spec: \op{\cring} \to \LCRS $ (which is unique up to
a unique isomorphism).

\begin{definition} \label{def:affine} An {\sf  affine
  $\cin$-scheme} is a local $\cin$-ringed space isomorphic
to $\Spec(\scA)$ for some $\cin$-ring $\scA$.

The {\sf category of
  affine schemes}  $\Aff$ is the essential image of the functor $\Spec$.
\end{definition}

\begin{notation} \label{not:2.7}
Given an affine scheme $\Spec(\scA)$ we denote the underlying
topological space by $X_\scA$. To avoid the clutter we  denote  its
structure sheaf by $\scO_\scA$ (and not by $\scO_{X_\scA}$ as we have
done for general local $\cin$-ringed space).  Thus
\[
\Spec(\scA) = (X_\scA, \scO_\scA).
\]  
\end{notation}

\begin{notation}
We denote the unit and the counit of the adjunction $\Gamma \vdash
\Spec$ by $\eta$ and $\varepsilon$, respectively.
\end{notation}
\subsection*{Complete $\cin$-rings}
Unlike the spectrum functor  in algebraic geometry $\Spec: \op{\cring} \to
\Aff$ is {\sf  not} an equivalence of categories and  $\Gamma \vdash
\Spec$ is not an adjoint equivalence.   
For instance, there are
nonzero $\cin$-rings with {\em no} $\R$-points
(Definition~\ref{def:A.17} and
Remark~\ref{rmrk:no-R-points}).  The spectrum of such a ring is  the
empty set with the zero $\cin$-ring; which is the same as $\Spec(0)$.
In other words, there are $\cin$-rings $\scA$ for which the components
$\varepsilon_\scA :\scA \to \Gamma (\Spec(\scA))$ of the counit of adjunction are
not isomorphisms.

However, there is a full
subcategory of $\op{\cring} $ consisting of the so called complete
rings (Definition~\ref{def:complete} below) so that the restriction of $\Spec$ to
this subcategory is part of an adjoint equivalence.

\begin{definition} \label{def:complete}(cf. \cite[Definition~4.35]{Joy}) A $\cin$-ring $\scA$ is {\sf
    complete} if the component $\varepsilon_\scA:\scA \to \Gamma
  (\Spec(\scA))$ of the counit of adjunction $\Gamma \vdash
\Spec$ is an isomorphism.
\end{definition}

\begin{remark}\label{rmrk:2.10} Dubuc \cite{Dubuc} showed
  (\!\!\cite[13.~Theorem]{Dubuc}) that if a $\cin$-ring $\scA$ is
  finitely generated (see Definition~\ref{def:fg}) and germ determined
  (Definition~\ref{def:germ_determined}) then $\varepsilon_\scA:\scA
  \to \Gamma (\Spec(\scA))$ is an isomorphism, i.e., that the ring $\scA$ is
  complete.  Note that if $\scA = \cin(\R^n)/J$ for some finitely
  generated ideal $J\subset \cin(\R^n)$, then $\scA$ is germ
  determined \cite{Dubuc}.\\

  The fact that any complete ring has to be germ determined is not
  hard to see: Joyce's construction of $\Spec$
  (Construction~\ref{construction:1} below), Lemma~\ref{lem:A.24} and
  Corollary~\ref{cor:A.25} imply that the  component of the counit
  $\varepsilon_\scA:\scA \to \Gamma(\Spec(\scA))$ is injective if and
  only if the $\cin$-ring $\scA$ is germ determined.  The surjectivity
  of $\varepsilon_\scA:\scA
  \to \Gamma (\Spec(\scA))$ is implied by the existence of partition
  of unity on $\R^n$ \cite{Joy}.
\end{remark}  

\begin{notation}
We denote the full subcategory of the category $\cring$ consisting of complete
$\cin$-rings by $\comp{\cring}$.
\end{notation}

It follows from the property of adjunctions (see, for example
\cite[Exercise~2.2.11]{Leinster}) that
\[
  \Spec|_{\comp{\cring}} : \op{(\comp{\cring})} \to \lcrs
\] is fully faithful and that the essential image of this restriction
consists of all locally ringed spaces such that the components
$\eta_{(X,\scO_X)}: (X,\scO_X) \to \Spec (\Gamma (X,\scO_X))$ of the
unit $\eta$ of adjunction are isomorphisms.  Joyce
(\cite[Theorem~3.36(a)]{Joy}) shows that this essential image is
exactly the category $\Aff$ of affine $\cin$-schemes
(Definition~\ref{def:affine}).  We thus get an adjoint equivalence
\[
   \left.\Gamma\right|_\Aff: \Aff \quad
 \substack{\longrightarrow\\[-1em] \longleftarrow} \quad \op{(\comp{\cring})}:
 \left. \Spec\right|_{\op{(\comp{\cring})}}
\]
given by restrictions of $\Spec$ and $\Gamma$,
respectively. 

In order to construct integral curves and flows of vector fields on
schemes we need to recall some details of the construction of the functor $\Spec$
given in \cite{Joy}.  Thus the rest of the section is a paraphrase of an argument in \cite{Joy}.
We start by reviewing  a construction of a
functor from $\cin$-rings to topological spaces.  Recall that 
an {\sf $\R$-point} of a $\cin$-ring $\scA$ is a map of $\cin$-rings $p:\scA
\to \R$ (Definition~\ref{def:A.17}).
\begin{notation} \label{not:2.13}
We denote the set of all $\R$-points of a $\cin$-ring $\scA$ by
$X_\scA$.
Thus
\[
X_\scA: =\{ p:\scA\to \R\mid p \textrm{ is a map of $\cin$-rings}\}
\equiv \Hom(\scA, \R).
\]  
\end{notation}
We will see shortly Notations~\ref{not:2.7} and \ref{not:2.13} are consistent.
\begin{remark} \label{rmrk:Milnors}
It is well-known that if $M$ is a (second-countable Hausdorff)
manifold then the map
\begin{equation} \label{eq:2.1}
M\to \Hom(\cin(M), \R),\qquad p\mapsto \ev_p
\end{equation}
is a bijection (here and below $\ev_p (f) = f(p)$ for all $f\in
\cin(M)$). This is a theorem of Pursell \cite[Chapter
8]{Pursell}. It is often referred to as Milnor's exercise. 
\end{remark}

\begin{construction}[The Zariski topology $\scT_\scA$ on the set
  $X_\scA$ of $\R$-points of a $\cin$-ring $\scA$] \label{constr:Z}
  The set $X_\scA$ of $\R$-points of a $\cin$-ring $\scA$ comes
  equipped with a natural topology, the Zariski topology.  It is
  defined as follows: for $a\in \scA$ let
\[
  U_a: = \{p\in X_\scA \mid p(a) \not = 0\}.
\]
Since for a point $p$ of $\scA$ and $a,b\in \scA$
\[
p(ab) \not = 0 \quad \Leftrightarrow \quad p(a) \not = 0 \textrm{ and
} p(b) \not = 0,
\]
we have $U_{ab} = U_a \cap U_b$. Hence the collection  $
\{ U_a\}_{a\in \scA}$ of sets
is a basis for a topology on $X_\scA$, which we denote by $\scT_\scA$.
\end{construction}

\begin{remark} \label{rmrk:2.14}
By a theorem of Whitney any closed subset $C$ of a smooth manifold $M$
is the set of zeros of some smooth function $f\in \cin(M)$.
It follows that the bijection \eqref{eq:2.1} is a homeomorphism
between the manifold $M$ with its given topology and the set
$X_{\cin(M)}$ of $\R$-points of $\cin(M)$ with the Zariski topology $\scT_{\cin(M)}$.
\end{remark}  

\begin{definition} \label{def:2.17}
Let $M$ be a manifold and $J\subset \cin(M)$ an ideal.  The {\sf zero set}  of the ideal $J$ is the subset $Z_J\subset M$ given by
\[
Z_J:= \{p \in M\mid j(p) = 0\textrm{ for all }j\in J\}.
\]
\end{definition}
\begin{remark}
Since by assumption the manifold $M$ is Hausdorff, the set $Z_J$ of Definition~\ref{def:2.17}
is closed  in $M$.
\end{remark}
The following lemma is a generalization  Remark~\ref{rmrk:2.14}.
\begin{lemma}  \label{lem:2.19} Let $M$ be a manifold, $J\subset \cin(M)$ the ideal and
$\scA = \cin(M)/J$ the quotient $\cin$-ring (cf.\ Lemma~\ref{lem:2.13}).    Then any point $p$
in the zero set $Z_J$ of $J$ gives rise to an $\R$-point
\[
\overline{ev}_p:\scA \to \R, \qquad \overline{ev}_p(f+J): = f(p).
\]  
The map
\[
\nu: Z_J\to X_\scA, \qquad \nu( p) := \overline{ev}_p
\]
is a homeomorphism (where $Z_J$ is given the subspace topology and
$X_\scA$ the Zariski topology).
\end{lemma}

\begin{proof}
Since the quotient map $\Pi:\cin(M) \to \scA = \cin(M)/J$ is onto, the
map
\[
  \Pi^*:X_\scA \to X_{\cin(M)}, \qquad \Pi^*(q) = q\circ \Pi
\]
  is injective.  By definition of the map
$\nu$,
\[
\nu(p) \circ \Pi = \ev_p.
\]
Therefore if $\nu(p) = \nu(q)$, then $\ev_p = \ev_q$.  Since the set
of smooth functions $\cin(M)$
separates points of the manifold $M$, %
the map  $\nu$ is injective.

Given $\psi \in X_\scA$, $\psi \circ \Pi \in \Hom(\cin(M), \R)$.  By
Milnor's exercise, $\psi \circ \Pi = \ev_p$ for some point $p\in M$.
If $p\not \in Z_J$ there is $j\in J$ so that $\ev_p(j) \not = 0$.
Contradiction, since
\[
\ev_p (j) = \psi (\Pi (j)) = \psi (0) = 0.
\]
Therefore $\nu$ is onto.

It remains to check that $\nu$ is a
homeomorphism.  Recall that the topology on $X_\scA$ is generated by
the sets of the form $U_a: = \{\psi\in X_\scA \mid \psi (a) \not =
0\}$, where $a\in \scA = \cin(M)/J$.   By a theorem of Whitney any
open subset of the manifold $M$ is of the form $\{f\not = 0\}$ for
some $f\in \cin(M)$ (cf.\ Remark~\ref{rmrk:2.14}).   Now given  $f\in \cin(M)$,
\[
\nu (\{f \not  =0\} \cap Z_J) = \{\overline{\ev}_p \in X_\scA \mid f(p)
\not = 0\} = U_{f+J}.
\]
And
\[
\nu\inv (U_{f+J}) = \{p\in Z_J \mid f(p) \not = 0\} = \{f\not = 0\}
\cap Z_J. 
\]
It follows that  $\nu$ is a homeomorphism.
\end{proof}

We note parenthetically that Lemma~\ref{lem:2.19} has a partial
converse.
\begin{lemma}
Let $M$ be a manifold, $C\subset M$ a closed subset and $I\subset
\cin(M)$ the vanishing ideal of $C$:
\[
I = \{ f\in \cin(M) \mid f(x) = 0 \textrm{ for all } x\in C\}.
\]
Then the zero set $Z$ of $I$ is $C$.
\end{lemma}
\begin{proof}
If $p\in C$ is a point then any function $f\in I$ vanishes on $p$
hence $p\in Z$.   Conversely if $p\not \in C$ then, since $C$ is
closed and $M$ is Hausdorff, there exists a function $f\in \cin(M)$ so
that $f(p) \not = 0$ and $f|_C = 0$.  Hence $p\not \in Z$.  
\end{proof}

To construct the functor $\Spec$ on objects, given a $\cin$-ring
$\scA$  we need to equip the 
topological space $X_\scA$ of its points with a sheaf $\scO_\scA$ of
$\cin$-rings. To construct this structure sheaf  $\scO_\scA$ we need to recall a
few things about localizations of $\cin$-rings.

\begin{lemma} \label{lem:2.21}
Given a $\cin$-ring $\scA$ and a set $\Sigma$ of nonzero elements of $\scA$
there exists a $\cin$-ring $\scA\{\Sigma\inv\}$ and a map $\gamma:
\scA\to \scA\{\Sigma\inv\}$ of $\cin$-rings with the following universal property:
\begin{enumerate}
\item $\gamma(a)$ is invertible in $\scA\{\Sigma\inv\}$ for all $a\in
  \Sigma$ and 
\item  for
any map $\varphi:\scA\to \scB$ of $\cin$-rings so that $\varphi(a)$ is
invertible in $\scB$ for all $a\in \Sigma$ there exists a unique map
$\overline{\varphi} :\scA\{\Sigma\inv\} \to \scB$ making the diagram
\[
  \xy
(-10,10)*+{\scA }="1";
(14,10)*+{\scB}="2";
(-10, -6)*+{\scA\{\Sigma\inv\}}="3";
{\ar@{->} ^{\varphi} "1";"2"};
{\ar@{->}_{\gamma} "1";"3"};
{\ar@{->}_{\overline{\varphi}} "3";"2"};
\endxy
\]  
commute.
\end{enumerate}
\end{lemma}
\begin{proof} See  \cite[p.~23]{MR}.
\end{proof}
\begin{definition}
  We refer to the map 
$\gamma:\scA \to \scA\{\Sigma\inv\}$ of Lemma~\ref{lem:2.21} as  a {\sf localization} of
the $\cin$-ring $\scA$ at the set $\Sigma$.   
\end{definition}
\begin{remark} A localization of a $\cin$-ring $\scA$ at a set
  $\Sigma$ is unique up to a unique isomorphism, so we can speak about
  {\em the} localization of $\scA$ at $\Sigma$.
\end{remark}

\begin{notation}\label{not:2.22}
Let $\scA$ be a $\cin$-ring. By Lemma~\ref{lem:2.21}, for an
$\R$-point $x:\scA\to \R$ of a $\cin$-ring $\scA$ there exists a
localization of $\scA$ at the set
\[
\{x\not =0\}:= \{a\in \scA \mid x(a) \not = 0\}
\]
We set
\[
\scA_x:= \scA\{ \{x\not =0\}\inv\}.
\] 
and denote the corresponding localization map by
\begin{equation} \label{eq:pix}
  \pi_x:\scA\to \scA_x
\end{equation}
\end{notation}
\noindent
Joyce proves \cite[Proposition~2.14]{Joy} that $\pi_x:\scA \to \scA_x$ is surjective
with $I_x:= \ker \pi_x$ given by
\begin{equation} \label{eq:5.13}
I_x:= \{ a \in \scA \mid \textrm{ there is } d\in \scA \textrm{ so
  that } x(d)\not =0 \textrm{ and } ad= 0\}.
\end{equation}
We think of $\scA_x$ as the ring of germs of elements of $\scA$ at the
point $x$.

\begin{remark} In case of $\scA = \cin(\R^n)$ and
  $x:\cin(\R^n) \to \R$ (which is the evaluation at a some point
  $p\in \R^n$ by Remark~\ref{rmrk:Milnors}) 
  the localization $(\cin(\R^n))_x$ is isomorphic to the usual ring of
  germs of functions at $p$.  This is because both rings are
  localizations of $\cin(\R^n)$ at the same set; see \cite[Example
  2.15]{Joy}.
\end{remark}
\begin{remark}
  The localizations $\scA_x$ are local rings.  This is easy to see.
  Note first that for any $a\in \{x\not =0\}$, $x(a)$ is invertible in
  $\R$, hence $x:\scA\to \R$ gives rise to $\overline{x}:\scA_x\to \R$
  with $ x = \overline{x}\circ \pi_x$.  Moreover, for any $c\in \scA$,
  $\pi_x(c) \not \in \ker{\overline{x}}$ if and only if
  $x(c) \not = 0$ if and only if $\pi_x(c)$ is invertible in
  $\scA_x$. Hence $\scA_x \smallsetminus \ker{\overline{x}}$ consists
  of units of $\scA_x$ and therefore $\ker{\overline{x}}$ is a unique
  maximal ideal in $\scA_x$.
\end{remark}

\begin{construction} \label{construction:1}  Given a $\cin$-ring
$\scA$ we construct the corresponding affine $\cin$-scheme
$\Spec(\scA)$.  In Construction~\ref{constr:Z} we defined  a
topology on the set $X_\scA$ of $\R$-points of $\scA$.  We now
construct the structure sheaf $\scO_\scA$ on $X_\scA$.

We start by defining  a candidate etale space $S_\scA$ of the
sheaf $\scO_\scA$:
\[
S_\scA:= \bigsqcup _{x\in X_\scA} \scA_x \equiv \bigsqcup _{x\in X_\scA} \scA/I_x.
\]
The set $S_\scA$ comes with the evident surjective map $\varpi:
S_\scA\to X_\scA$ defined by $\varpi( s) = x$ for all $s\in
\scA_x\hookrightarrow S_\scA$. For any $a\in \scA$ we get a section $\frs_a: X_\scA
\to S_\scA$ of $\varpi$:
\begin{equation} \label{eq:2.6}
\frs_a(x) = \pi_x (a) \equiv a_x,
\end{equation}
where, as before, $\pi_x:\scA \to \scA_x$ is the localization map
\eqref{eq:pix}.  The collection of sets
\[
\{ \frs_a (U)\mid a\in \scA, U\in \Open(X_\scA) \}
\]  
forms a basis for a topology on $S_\scA$.  In this topology the
projection $\varpi: S_\scA\to X_\scA$ is a local homeomorphism and all
sections $\frs_a:X_\scA \to S_\scA$ are continuous.

{\em We
define the structure sheaf $\scO_\scA$ of $\Spec(\scA)$ to be the sheaf
of continuous sections of  $\varpi: S_\scA \to X_\scA$.}     Equivalently
\begin{equation} \label{eq:2.4}
\begin{split}  
\scO_\scA (U) =   &\{ s:U\to \bigsqcup _{x\in U} \scA_x\mid \textrm{ there
  is an open cover } \{U_\alpha\}_{\alpha \in A} \textrm{ of } U\\ 
  &
\textrm{ and } \{a_\alpha\}_{\alpha\in A} \subset \scA \textrm{ so that }
s|_{U_\alpha} = \frs_{a_\alpha}|_{U_\alpha} \textrm{ for all } \alpha \in A\}.
\end{split}
\end{equation}
for every open subset $U$ of $X_\scA$.  The $\cin$-ring structure on
the sets $\scO_\scA(U)$ is defined pointwise.

It follows from \eqref{eq:2.4} that the sheaf $\scO_\scA$ is a
sheafification of a presheaf $\cP_\scA$ defined by
\begin{equation}\label{eq:sec}
  \cP_\scA(U) := \{\frs_a|_U \mid a\in \scA\}.
\end{equation}

It turns out that the stalk of the sheaf $\scO_\scA$ at a point
$x\in X_\scA$ is (isomorphic to) $\scA_x$ (\!\cite[Lemma~4.18]{Joy}).  Consequently the pair
$(X_\scA, \scO_\scA)$ is a {\em local} $\cin$-ringed space. 
\end{construction}

\begin{remark} \label{rmrk:2.26}
Note that the canonical map $\eta: \cP_\scA\to \scO_\scA$ from a presheaf to its
sheafification is simply the inclusion: $\eta(\frs_a) = \frs_a$.

Note also that for any open subset $U$ of $X_\scA$ the map
\[
\scA \to \cP(U), \qquad a\mapsto \frs_a|_U
\]
is a surjective map of $\cin$-rings with the kernel $\bigcap_{x\in
  U}I_x$.  Thus
\[
\cP(U) \simeq \scA/\bigcap_{x\in U}I_x.
\]  
\end{remark}  
\begin{remark} \label{rmrk:2.27}
In the case where the $\cin$-ring $\scA$ in
Construction~\ref{construction:1} is the ring $\cin(M)$ of smooth
functions on a manifold $M$, the space $X_\scA$ of points of $\scA$  is canonically
homeomorphic to the topological space $M$ underlying the manifold $M$
--- cf.\ Lemma~\ref{lem:2.19}.  More generally, as a local
$\cin$-ringed space $\Spec(\cin(M))$ is isomorphic to $(M, \cin_M)$
where $\cin_M$ denotes the sheaf of smooth functions on the manifold
$M$.   Moreover there is a unique  isomorphisms
$\uu{\varphi}_M: (M, \cin_M) \to \Spec(\cin(M))$ that makes the diagram
\[
\xy
(-10,10)*+{\cin(M)}="1";
(14,10)*+{\Gamma(\cin_M)}="2";
(14, -10)*+{\Gamma(\scO_{\cin(M)})}="3";
{\ar@{->} ^{\id} "1";"2"};
{\ar@{->}_{\varepsilon_{\cin(M)}} "1";"3"};
{\ar@{<-}_{\Gamma{(\uu{\varphi})}} "3";"2"};
\endxy
\]
commute.   From now on we suppress the isomorphism
$\uu{\varphi}_M: (M, \cin_M) \to \Spec(\cin(M))$ for all manifolds
$M$.  That is, in effect, we set
\[
\Spec(\cin(M)) = (M, \cin_M)
\]  
for all manifolds $M$.
\end{remark}
\begin{remark} \label{rmrk:corners}
Similarly, if $M$ is a manifold with corners then $\Spec(\cin(M))$ is
(isomorphic to) $(M, \cin_M)$, where as above $\cin_M$ is the sheaf of
smooth functions on $M$.
\end{remark}  

\begin{remark} \label{rmrk:Dubuc-fg}
  
In \cite{Dubuc} on p.~687 in the paragraph right below Proposition~12 
Dubuc describes the affine scheme $\Spec(\scA)$ for a finitely
generated $\cin$-ring $\scA$ as follows. Since $\scA$ is finitely
generated there is a surjective map of $\cin$-rings $\Pi: \cin(\R^n)
\to \scA$.  Let $J$ denote the kernel of $\Pi$ and let
\begin{equation} \label{eq:A.vanish}
Z = Z_J:= \{p\in \R^n \mid f(p) = 0 \textrm{ for all } f\in J\}
\end{equation}
denote the zero set of the ideal $J$.   For each open set $U\subset
\R^n$ we then have the quotient $\cin$-ring
\[
\cP(U) = \cin(U)/J_U
\]  
where $J_U \subset \cin(U)$ is the ideal generated by the set $J|_U =\{f|_U\mid f\in
J\}$.  Note that if $U\cap Z_J= \varnothing$, then $\cP(U) = 0$.  Let
$\cE$ be the sheafification of $\cP$.  The support of $\cE$ is $Z_J$,
hence we can view $\cE$ as a sheaf on $Z$.   Dubuc then tells us that  $\Spec(\scA) \simeq
(Z, \cE)$.
\end{remark}

\begin{remark} \label{rmrk:2.28}
It follows from Construction~\ref{construction:1}  that for any
$\cin$-ring $\scA$ we have a map
\[
\Psi_\scA: \scA \to \Gamma(\Spec(\scA)), \qquad a \mapsto \frs_a.
\]  
Joyce proves \cite[p.\ 28]{Joy} that the maps $\{\Psi_\scA\}_{\scA\in
  \cring}$ assemble into a natural transformation $\Psi: \Gamma \circ
\Spec \Rightarrow \id_{\op{\cring}}$ and that this natural
transformation is the counit $\varepsilon$ of the adjunction $\Gamma
\vdash \Spec$.  Thus for a $\cin$-ring $\scA$ the component of the
adjunction $\varepsilon_\scA: \scA \to \Gamma (\Spec (\scA)) $ is
given by 
\begin{equation}
\varepsilon_\scA (a) = \frs_a
\end{equation}
for all $a\in \scA$, where $\frs_a : X_\scA \to S_\scA$ is defined by \eqref{eq:2.6}.
\end{remark}

Next we define the functor $\Spec$ on morphisms.  To do this we need a lemma.

\begin{lemma} \label{lem:2.25}
Let $\varphi: \scA \to \scB$ be a map of $\cin$-rings and $x:\scB\to
\R$ a point.  Then $\varphi$ induces a unique map $\varphi_x:
\scA_{x\circ \varphi} \to \scB_x$ of $\cin$-rings on their respective  localizations making the diagram
\begin{equation} \label{eq:2.4n}
  \xy
(-10,6)*+{\scA }="1";
(14,6)*+{\scB}="2";
(-10, -6)*+{\scA_{x\circ \varphi}}="3";
 (14, -6)*+{\scB_x}="4";
{\ar@{->}_{\pi_{x\circ \varphi}} "1";"3"};
{\ar@{->}^{\pi_x} "2";"4"};
{\ar@{->}^{\varphi} "1";"2"};
{\ar@{->}_{\varphi_x} "3";"4"};
\endxy
\end{equation}
commute.
\end{lemma}

\begin{proof}
Since $\pi_y: \scA \to \scA_y$ is a localization of the $\cin$-ring
$\scA$ at the set $\{y\not =0\} := \{a\in \scA \mid y(a)\not =0\}$ and
$\pi_x:\scB \to \scB_x$ is the localization of $\scB$ at $\{x\not
=0\}$, it is enough to show that
\begin{equation} \label{eq:2.3}
  \varphi(\{x\circ \varphi\not =0\}) \subseteq
  \{x\not = 0\}.
\end{equation}
Since
\[
(x \circ \varphi)(a) \not = 0 \quad \Leftrightarrow \quad x ( \varphi(a)) \not =0,
\]
\eqref{eq:2.3} holds and we are done.
\end{proof}

\begin{construction}[Construction of $\Spec$ on morphisms] \label{construction:2}
Let $\varphi:\scA \to \scB$ be a map of $\cin$-rings.   Define a map
of sets $f(\varphi):X_\scB \to X_\scA$ ( where as before
$X_\scA= \Hom (\scA,\R)$ and $X_\scB= \Hom(\scB, \R)$) by
\[
f(\varphi)(x):= x\circ \varphi.
\]
Recall that $\{ V_a = \{ y\in X_\scA \mid y(a) \not =0\}\}_{a\in
  \scA}$ is a basis for the topology on $X_\scA$ and similarly $\{ U_b
= \{ x\in X_\scB \mid x(b) \not =0\}\}_{b\in   \scB}$ is a basis for
the topology on $X_\scB$.  Since
\[
(f(\varphi))\inv (V_a) = \{ x \in X_\scB\mid f(\varphi) x \in V_a\} =
\{ x \in X_\scB \mid x (\varphi(a)) \not =0\} = U _{\varphi(a)},
\]
the map $f(\varphi)$ is continuous with respect to the Zariski
topologies on $X_\scB$ and $X_\scA$.  It remains to construct a map of
sheaves
\[
f(\varphi)_\#: \scO_\scA \to f(\varphi)_* \scO_\scB.
\]
Fix an open set $U\subseteq X_\scA$.  We construct a map 
\[
f(\varphi)_{\#, U}: \cP_\scA (U)\to \scO_\scB (f(\varphi)\inv U).
\]
of $\cin$-rings as follows.  Recall that
\[
  \cP_\scA (U) = \{ \frs_a: U\to \bigsqcup _{x\in U} \scA_x\mid a\in
  \scA\}
\]
where, as in \eqref{eq:2.6}, $\frs_a(y) = a_y$ for all $y\in U$.  Here
and below to reduce the clutter we write $\frs_a$ when we mean $\frs_a|_U$.
Given $\frs_a \in \cP_\scA (U)$ consider
\[
\ti{\frs}_a : f(\varphi)\inv (U) \to \bigsqcup _{x\in f(\varphi)\inv
  (U)} \scB_x, \qquad \ti{\frs}_a (x)  := \varphi_x (\frs_a (f(\varphi)x)),
\]
where $\varphi_x: \scA_{f(\varphi)x} \to \scB_x$ is the map from
Lemma~\ref{lem:2.25}.   Note that 
\[
\varphi_x (\frs_a (f(\varphi)x)) = \varphi_x (\pi_{f(\varphi) x} a)
= \pi_x (\varphi(a)),
\]  
where the last equality is commutativity of \eqref{eq:2.4n}.   Hence
$\ti{\frs}_a  = \frs_{\varphi(a)} \in \cP_\scB(f(\varphi)\inv (U))
\subset \scO_\scB(f(\varphi)\inv (U))$.  We therefore define
\[
  f(\varphi)_{\#, U}: \cP_\scA (U)\to \scO_\scB (f(\varphi)\inv U)
\]
by
\[
f(\varphi)_{\#, U} (\frs_a): = \frs_{\varphi(a)}
\]  
for all $a\in \scA$.  It is easy to see that the family of maps
$\{f(\varphi)_{\#, U}\}_{U\subset X_\scA}$ is a map of presheaves
$f(\varphi)_\#: \cP_\scA \to f(\varphi)_* \scO_\scB$.  By the
universal property of sheafification we get a map of sheaves
$f(\varphi)_\#: \scO_\scA \to f(\varphi)_* \scO_\scB$.

We define
\[
  \Spec(\varphi) \equiv
  \uu{f(\varphi)}: = (f(\varphi), f(\varphi)_\#).
\]  
\end{construction}

\section{$\cin$-derivations and $\Spec$} \label{sec:3}

The goal of this section is to prove a number of technical results that are 
then used in Section~\ref{sec:int} to prove Theorem~\ref{thm:4.1}.
Recall that the  set $\CDer(\scA)$  of $\cin$-derivations of a
$\cin$-ring $\scA$ (Definition~\ref{def:der2} and Notation~\ref{not:der})
is an $\scA$-module: the action of $\scA$ on on $\CDer(\scA)$ is given
by 
\[
(av ) (b): = a (v(b))
\]
for all $a\in \scA$, $v\in \CDer(\scA)$ and $b\in \scA$.

Similarly the set $\cX(X,\scO_X)$ of vector fields on a local
$\cin$-ringed space $(X, \scO_X)$
(Notation~\ref{notation:presheaf_maps}, Definition~\ref{def:vf},
Notation~\ref{not:vf}) is a $\Gamma (\scO_X)$-module: given
$v\in \cX(X,\scO_X)$, $a\in \Gamma (\scO_X)$ we define
$av: \scO_X\to \scO_X$ by
\[
(av)_U (s) := a|_U v_U(s)
\]  
for all open sets $U\subset X$ and all sections $s\in \scO_X(U)$.

If
$(X, \scO_X)$ is the affine scheme  $\Spec(\scA)$ 
then the counit of adjunction $\varepsilon_\scA:\scA\to \Gamma
(\Spec (\scA))$ makes the set $\cX(\Spec(\scA))$ into an
$\scA$-module.
We need to understand the relation between $\cin$-derivations on
$\cin$-rings and vector fields on affine $\cin$-schemes.\\

\begin{lemma} \label{lem:2.37}
Let $v:\scA\to \scA$ be a $\cin$-derivation of a $\cin$-ring
$\scA$ and $I\subseteq \scA$ an ideal which is preserved by $v$: $v(I)\subset I$. Then the
induced map
\[
\overline{v}:\scA/I\to \scA/I, \qquad \overline{v}(a+I): = v(a)+ I
\]  
is a $\cin$-ring derivation.
\end{lemma}

\begin{proof}
Given $n$, $f \in \cin(\R^n) $ and $a_1,\ldots, a_n \in \scA$ we
compute:
\[
\begin{split}
\overline{v} (f_{\scA/I} (a_1 +I,\ldots, a_n+I) =&\overline{v}( f_\scA(a_1,\ldots,a_n) +I)\\
=&v(f_\scA(a_1,\ldots,a_n)) +I\\
=& \left( \sum (\partial_i f)_\scA(a_1,\ldots, a_n) v(a_i) \right)+I\\
  =&  \sum( (\partial_i f)_\scA(a_1,\ldots, a_n) +I)( v(a_i) +I)\\
=&  \sum (\partial_i f)_{\scA/I} (a_1+I,\ldots, a_n+I) \overline{v}(a_i+I).
\end{split}
\]  
\end{proof}

\begin{example} \label{ex:3.2}
Let $N$ be a manifold, $U\subset N$ a nonempty open subset, $M$ the
closure of $U$ in $N$ and $\scA = \cin(N)/I_M$ where $I_M$ is the vanishing
ideal of $M$:
\[
I_M :=\{f\in \cin(N) \mid f(x) = 0\textrm{ for all } x\in M\}. 
\]
It is common to think of $\scA$ is a ring of smooth functions on the
space $M$ and write $\cin(M)$ for $\scA$.
Note that 
\[
I_M = \{f\in \cin(N)\mid f|_U =0\}
\]
since $U$ is dense in
$M$.  We claim that any vector field $v$ on $N$ (which is a
$\cin$-ring derivation of $\cin(N)$ preserves the ideal $I_M$.   This
is because if $f\in \cin(N)$ vanishes identically on the open set $U$
then so does $v(f)$.

Therefore any vector field $v$ on $N$ induces a $\cin$-ring derivation
$\overline{v}: \scA \to \scA$.
\end{example}

\begin{example}
Consider the ideal $I:= (x^2y) \subset \cin(\R^2)$ generated by the
function $x^2y$.   The zero set $Z$ of the ideal $I$  (cf.\
Definition~\ref{def:2.17}) is the union of the two coordinate axis.
Then
\[
(x\partial_x + y\partial_y) \, (x^2y) = 2x^2y + x^2 y \in I.
\]  
Hence $v= x\partial_x + y\partial_y$ preserves $I$ and induces a
derivation $\overline{v}: \scA \to \scA$ where $\scA = \cin(\R^2) /I$.

Note that $xy\not \in I$ but $(xy)^2 \in I$.  Hence $\scA$ is not a
ring of functions on the zero set $Z$ of $I$ (or on any set).
\end{example}

\begin{example} \label{ex:3.4}Let $v:\scA\to \scA$ be a $\cin$-ring derivation of a
$\cin$-ring $\scA$ preserving an ideal $I\subset \scA$.  Then the
square $I^2$ of $I$ is also preserved by $v$.  This is because for any
$a,b\in I$,
\[
v(ab) = v(a) b + av(b) \in I^2.
\]
\end{example}  
\begin{example}  Consider the standard (middle third) Cantor set
$C\subset \R$ and let $I $ be the vanishing ideal of $I$ in
$\cin(\R)$.  Recall that $C$ is perfect.  That is, for any $x_0\in C$ 
any open neighborhood $U$ of $x_0$ in $\R$ contains $x\in C$ with
$x\not = x_0$.   Then if $f \in \cin(\R)$ vanishes identically on $C$
then so does its derivative $f'$.  This is because if there is $x_0\in
C$ with $f'(x_0) \not = 0$ then there is an open neighborhood $U$ of
$x_0$ in $\R$ so that $x_0$ is the only zero of $f|_U$.  Since $C$ is
perfect this cannot happen.

It follows that the derivation $\frac{d}{dx}$ of $\cin(\R)$ preserves
the vanishing ideal $I$ of the Cantor set hence induces a $\cin$-ring
derivation on $\scA = \cin (\R)/I =: \cin(C)$.

Note that for any function $f\in \cin(\R)$, $f\frac{d}{dx}$ also
preserves the vanishing ideal $I$ and induces a derivation on
$\cin(C)$.

Note also that by Example~\ref{ex:3.4} any vector field
$f\frac{d}{dx}$ preserves the square $I^2$ of the vanishing ideal $I$
of the Cantor set, hence induces a $\cin$-derivation on $\cin(\R)/I^2$.
\end{example}  

\begin{notation} \label{not:der_i}
Let $I$ be an ideal in a $\cin$-ring $\scA$.   Denote the set of all
$\cin$-ring derivations of $\scA$ that preserve $I$ by
$\CDer(\scA)_I$.  That is
\[
\CDer(\scA)_I := \{v\in \CDer(\scA)\mid v(I) \subset I\}.
\]  
\end{notation}

\begin{remark}  In light of Notation~\ref{not:der_i},
  Lemma~\ref{lem:2.37} says that given an ideal $I$ in a $\cin$-ring
  $\scA$ we have a map
\[
\overline{\phantom{x}}: \CDer(\scA)_I \to \CDer(\scA/I), \quad v\mapsto \overline{v}.
\]    
It is easy to see that the map $\overline{\phantom{x}}$ is $\R$-linear
and sends commutators of derivations to commutators.  That is, it's a
map of real Lie algebras.

If $\scA = \cin(\R^n)$ the map $\overline{ \phantom{x}}$ is surjective;
see the next lemma.
\end{remark}

\begin{lemma} \label{lem:4.2} Let $I\subset \cin(\R^n)$ be an ideal,
  $\scA: = \cin(\R^n)/I$ the quotient $\cin$-ring and $\Pi:
  \cin(\R^n)\to \scA$ the quotient map.   Then for any derivation
  $\hat{v}\in \CDer(\scA)$ there is a derivation $V\in \cin(\R^n)$ so
  that
  \[
\hat{v} (f+I) = V(f)+I
\]
for all $f\in \cin(\R^n)$.   That is, $\overline{\phantom{x}}:
\CDer(\cin(\R^n))_I \to \CDer(\cin(\R^n)/I)$ is surjective.  
\end{lemma}

\begin{proof} 
Let $x_1,\ldots, x_n :\R^n
\to \R$ denote the standard coordinate functions. 
 Since $\Pi $ is
onto, for any $i$, $1\leq i\leq n$,
\begin{equation}
\hat{v}(\Pi(x_i)) = \Pi(a_i) 
\end{equation}
for some $a_i\in \cin(\R^n)$.  Let
\begin{equation} \label{eq:3.2}
V : = \sum_i a_i \frac{\partial}{\partial x_i}.
\end{equation}
Then $V$ is a $\cin$-ring derivation of $\cin(\R^n)$. Since the
coordinate functions  $x_1,\ldots,
x_n\in \cin(\R^n)$ generated $\cin(\R^n)$ as a $\cin$-ring
(cf.\ Remark~\ref{rmrk:free})
and 
since $\Pi$ is onto, $\Pi(x_1), \ldots,
\Pi(x_n)$ generate the $\cin$-ring $\scA$ (cf.\ Definition~\ref{def:fg}).
Observe that
\[
\Pi (V(x_i)) = \Pi (\sum_j a_j \frac{\partial}{\partial x_j} x_i)
=\Pi(a_i) = \hat{v}(\Pi (x_i))
\]
for all $i$.  
Since $V$ and $\hat{v}$ are $\cin$-ring derivations
(Definition~\ref{def:der2}) and since $x_1,\ldots,
x_n\in \cin(\R^n)$ freely generated $\cin(\R^n)$ it follows that
\[
\Pi (V(f)) = \hat{v}(\Pi (f))
\]
for all $f\in \cin(\R^n)$.
\end{proof}

\begin{lemma} \label{lem:2.18}
Let $x:\scA \to \R$ be an $\R$-point of a $\cin$-ring $\scA$, $\scA_x$
the localization of $\scA$ at the set $\{x\not= 0\}$ (see
Lemma~\ref{lem:2.21} and subsequent discussion).   A $\cin$-ring
derivation $v:\scA \to \scA$  preserves the ideal $I_x:= \ker(\scA
  \xrightarrow{ \pi_x} \scA_x)$ hence induces a derivation $v_x:\scA_x\to \scA_x$.
\end{lemma}

\begin{proof} Recall  that the  localization map
  $\pi_x:\scA \to \scA_x$ is surjective with the kernel
\[  
I_x: = \{a\in \scA \mid \textrm{ there is }d\in \scA \textrm{ so that }
  x(d)\not = 0\textrm{ and } ad=0\}.
\]  
We want to show: if $a\in I_x$ then $v(a) \in I_x$ as well (for then
we can define $v_x:\scA_x\to \scA_x$ by $v_x (b+I_x) := v(b)+ I_x$).
Since $a\in I_x$ there is $d\in \scA$ so that $x(d)\not = 0$ and
$ad=0$.   Since $v$ is a derivation
\[
0 = v(ad) = a v(d) + d v(a).
\]  
Since $I_x$ is an ideal, $a v(d) \in I_x$. Hence $d v(a) = -a v(d) \in
I_x$.  Therefore there is $d'\in \scA$ such that $x(d') \not = 0$ and
$0= d'(dv(a))  = (d'd)\, v(a)$.  Since $x$ is a homomorphism,  $0 \not =
x(d') x(d)  = x(d'd) $.  Therefore $v(a) \in I_x$.  Now apply Lemma~\ref{lem:2.37}.
\end{proof}

\begin{lemma}
Let $\scA$ be a $\cin$-ring and $\Spec(\scA)= (X_\scA, \scO_\scA)$ the
corresponding affine scheme.  The map
\begin{equation} \label{eq:vGamma}
\Gamma: \cX(\Spec(\scA)) \to \CDer(\Gamma(\Spec(\scA)) \equiv
\CDer(\scO_\scA(X_\scA)), \qquad \Gamma(v) := v_{X_\scA}
\end{equation}
is injective (Notation~\ref{notation:presheaf_maps}).
\end{lemma}  

\begin{proof}
It is clear that that the map $\Gamma$ in \eqref{eq:vGamma} is
linear.  If $\Gamma (v)= 0$, then 
$v_{X_\scA} (\frs_a)=0$ for any $a\in \scA$.   Since $v$ is a map of
presheaves of vector spaces, for any $U\in \Open(X_\scA)$,
\[
0 = v_{X_\scA} (\frs_a)|_U = v_U(\frs_a|_U).
\]
Therefore the restriction of $v$ to the presheaf $\cP_\scA \subseteq
\scO_\scA$ is zero.  Since $v$ is a map of presheaves and the
sheafification of $\cP_\scA$ is $\scO_\scA$ (both as a sheaf of
$\cin$-rings and as a sheaf of vector spaces), $v$ itself has to be 0.
\end{proof}

\begin{lemma} \label{lem:gamma-spec} \label{lem:3.4}
Let $\scA$ be a complete $\cin$-ring (so that the map
$\varepsilon_\scA:\scA \to \Gamma (\Spec(\scA))$ is an isomorphism).
Let $\Spec(\scA) = (X_\scA, \scO_\scA)$ denote the corresponding
affine scheme.     Then the map
\begin{equation} \label{eq:3.2+}
\bGamma: \cX (\Spec(\scA)) \to \CDer(\scA), \quad \bGamma(v) =
\varepsilon_\scA \inv \circ \Gamma (v) \circ \varepsilon_\scA
\end{equation}
(where $\Gamma(v)$ is given by \eqref{eq:vGamma})
is an isomorphism of real vector spaces, of real Lie algebras and
of $\scA$-modules.
\end{lemma}  

\begin{proof}
We prove that $\bGamma$ is a linear isomorphism.  The two other claims
follow easily from this fact.  It is also clear that $\bGamma$ is
$\R$-linear.
Since $\varepsilon_\scA$ is an isomorphism of $\cin$-rings and since
the $\Gamma$ is injective, the map $\bGamma$ is injective.

We prove that $\bGamma$ is surjective by constructing a linear map
\begin{equation} \label{eq:Spec}
\Spec: \CDer(\scA) \to \cX(\Spec(\scA))
\end{equation}  
with the property that $\bGamma (\Spec(w)) = w$ for all $w\in
\CDer(\scA)$.
The assumption that $\varepsilon_\scA:\scA \to \Gamma(\scO_\scA)$ is
an isomorphism implies, in particular, that every global section of
the structure sheaf $\scO_\scA$ is of the form $\frs_a$ for some $a\in
\scA$ (cf.\ Remark~\ref{rmrk:2.28}).

Now fix a derivation $w\in \CDer(\scA)$.  By Lemma~\ref{lem:2.18}
$w(I_x) \subset I_x$ for any point $x\in X_\scA$.   Hence for any $U\in
\Open(X_\scA)$ we have a well-defined $\cin$-ring derivation
\[
w_U: \scA/\bigcap_{x\in U}I_x \to  \scA/\bigcap_{x\in U}I_x , \qquad
w_U (a + \bigcap_{x\in U}I_x  ) = w(a) +\bigcap_{x\in U}I_x .
\]  
By Remark~\ref{rmrk:2.26} the map
\[
\scA/\bigcap_{x\in U}I_x  \to \cP_\scA(U), \qquad a + \bigcap_{x\in U}I_x 
\mapsto \frs_a|_U
\]  
is an isomorphism of $\cin$-rings.  It follows that
\[
v_U:\cP_\scA(U) \to \cP_\scA(U), \qquad v_U (\frs_a|_U) := \frs_{w(a)} |_U
\]  
is a well-defined $\cin$-ring derivation.    If $W, U\in
\Open(X_\scA)$ are two open sets with $W\subset U$ then
\[
 \left( v_U(\frs_a|_U)\right)|_W = \left(\frs_{w(a)}|_U\right) |_W =\frs_{w(a)}|_W
  = v_W(\frs_a|_W) =v_W((\frs_a|_U)|_W).
\]  
Hence the collection of derivations $\{v_U\}_{U\in \Open(X_\scA)}$ is
a map of presheaves of real vector spaces.  Since the structure sheaf
$\scO_\scA$ is a sheafification of $\cP_\scA$, there is a unique map
of sheaves
\[
 \Spec(w): \scO_\scA \to \scO_\scA
\]
so that for any open set $U\subset X_\scA$
\[
\Spec(w)_U (\frs_a|_U) = v_U(\frs_a|_U) =  s_{w(a)} |_U.
\]
In particular if $U= X_\scA$ then $\Spec(w)_{X_\scA} = \Gamma
(\Spec(w))$ satisfies
\begin{equation} \label{eq:3.4}
\Gamma (\Spec(w) )\frs_a  = \frs_{w(a)}.
\end{equation}  
Since $\varepsilon_\scA (a)  = \frs_a$ \eqref{eq:3.4} is equivalent to
\[
((\Gamma \circ \Spec) (w) )\circ \varepsilon_\scA = \varepsilon _\scA
\circ w.
\] 
Equivalently
\[
(\Gamma ( \Spec (w) )= \varepsilon_\scA \circ w \circ \varepsilon_\scA\inv .
\]  
Hence
\[
w = (\varepsilon_\scA\inv \circ \Gamma \circ \varepsilon_\scA)
(\Spec(w)) =\bGamma (\Spec(w)).
\]  
Consequently $\bGamma$ is surjective and we are done.  
\end{proof}

\begin{remark} \label{rmrk:3.5}
  It follows from the proof of Lemma~\ref{lem:gamma-spec} above that
  the map $\Spec: \CDer(\scA) \to \cX(\Spec(\scA))$ is the inverse of
  the map
  $\bGamma: \cX (\Spec(\scA)) \to \CDer(\scA)$.
\end{remark}  

\begin{remark} \label{rmrk:3.66}
Recall that given a $\cin$-manifold $M$ (with or without boundary)  we may view the local
$\cin$-ringed space $(M, \cin_M)$ as $\Spec(\cin(M))$
(Remark~\ref{rmrk:2.27}).  In differential geometry one usually views
a vector field $v$ on a manifold $M$ as either a ($\cin$-ring) derivation of
$\cin(M)$ or as a section of the tangent bundle $TM\to M$.  The latter
is equivalent to viewing $v$ as a map of sheaves of vector spaces:
\[
  v:\cin_M\to \cin_M
\]
so that for each $U\in
\Open(M)$ the map  $v_U:\cin(U) (\equiv \cin_M(U)) \to \cin(U)$ is a derivation.  In other words the map $\bGamma: \cX (\Spec(\cin(M)) \to
\CDer(\cin(M)$ and its inverse $\Spec$ are suppressed.   We will
adhere to  this tradition as far as manifolds are concerned.
\end{remark}  

\begin{lemma} \label{lem:3.7}
  Let $\scA$, $\scB$ be two complete $\cin$-rings 
  (so that the maps $\varepsilon_\scA:\scA\to \Gamma (\Spec(\scA))$, $\varepsilon_\scB:\scA\to \Gamma (\Spec(\scB)$ are isomorphisms and
  $\Spec: \Hom(\scB, \scA) \to \Hom (\Spec(\scA), \Spec(\scB))$ is a
  bijection).  
  
  For any map $\varphi:\scB\to \scA$ of $\cin$-rings
  and for any two derivations $w \in \CDer(\scB)$, $v\in \CDer(\scA)$,
  the derivations 
  $w$ and $v$ are $\varphi$-related if and only if the vector fields 
  $\Spec(v)   $ and $\Spec(w)$ (cf.\ \eqref{eq:Spec})
  are $\uu{f} := \Spec(\varphi)$-related.
\end{lemma}  

\begin{proof}
By Definition~\ref{def:1.4}  the vector fields $\Spec(v)$ and $\Spec(w)$ are $\uu{f}
= (f, f_\#)$-related if and only if the diagram
\begin{equation} \label{eq:3.5}
  \xy
(-10,6)*+{\scO_\scB }="1";
(14,6)*+{f_\# \scO_\scA}="2";
(-10, -6)*+{\scO_\scB}="3";
 (14, -6)*+{f_* \scO_\scA}="4";
{\ar@{->}_{\Spec(w)} "1";"3"};
{\ar@{->}^{f_*\Spec(v)} "2";"4"};
{\ar@{->}^{f_\#} "1";"2"};
{\ar@{->}_{f_\#} "3";"4"};
\endxy
\end{equation}
commutes.
Note that since the $\cin$-ring $\scA$ is complete, the map
\[
\Gamma: \cX(Spec(\scA)) \to \CDer(\Gamma(\Spec(\scA))
\]  
is an isomorphism.   Similarly the map $\Gamma: \cX(Spec(\scB)) \to
\CDer(\Gamma(\Spec(\scB))$ is an isomorphism.  Therefore
\eqref{eq:3.5} commutes if and only if
\begin{equation} \label{eq:3.6}
  \xy
(-14,10)*+{\Gamma(\scO_\scB) }="1";
(14,10)*+{\Gamma(\scO_\scA)}="2";
(-14, -6)*+{\Gamma(\scO_\scB)}="3";
 (14, -6)*+{\Gamma(\scO_\scA)}="4";
{\ar@{->}_{\Gamma(\Spec(w))} "1";"3"};
{\ar@{->}^{\Gamma(\Spec(v))} "2";"4"};
{\ar@{->}^{\Gamma(f_\#)} "1";"2"};
{\ar@{->}_{\Gamma(f_\#)} "3";"4"};
\endxy
\end{equation}
commutes. 
Consider the diagram 
\begin{equation}
\begin{tikzcd}[
  row sep=large,
  column sep=large,
]
\scB\ar[rr,"\varepsilon_\scB"]\ar[dd,"\varphi", swap]\ar[dr,"w"] &
&\Gamma(\scO_\scB)\ar[dr,"\Gamma(\Spec(w))"]\ar[dd,"\Gamma(f_\#)",near end] & \\
&\scB\ar[rr,"\varepsilon_\scB",near start,crossing over] %
& & \Gamma(\scO_\scB)  & \\
\scA\ar[rr,"\varepsilon_\scA",near start]\ar[dr,"v"]
& & \Gamma(\scO_\scA)\ar[dr,"\Gamma(\Spec(v))", near start] & \\
& \scA\ar[rr,"\varepsilon_\scA"]\ar[<-,uu,"\varphi",crossing over, near start,
swap] & & \Gamma(\scO_\scA)\ar[<-,uu,"\Gamma(f_\#)",swap] &
\end{tikzcd} .
\end{equation}
The back and front faces of the cube commute because $\varepsilon$ is a natural
transformation.  The top and bottom faces of the cube commute by construction of
$\Spec(w)$ and $\Spec(v)$, respectively (see proof of
Lemma~\ref{lem:gamma-spec} above).  Since $\varepsilon_\scA$, $\varepsilon_\scB$ are isomorphisms, it follows that  the left face of the
cube commutes if and only if the right face of the cube commutes.
\end{proof}
\noindent Essentially the same proof also gives us the following result. 

\begin{lemma} \label{lem:3.8}
Let
$\scA$, $\scB$ be two complete $\cin$-rings,  $\uu{f}: \Spec(\scB)\to
\Spec(\scA)$ a map of local $\cin$-ringed spaces.  Two vector fields
$v\in \cX(\Spec(\scA))$ and $w\in \cX(\Spec(\scB))$ are
$\uu{f}$-related if and only if $\bGamma(w)$ and $\bGamma(v)$ are
$\varphi$-related, where $\varphi = \Spec\inv (\uu{f})$ (here as
before $\bGamma: \cX(\Spec(\scA)) \to \CDer(\scA)$, $\bGamma:
\cX(\Spec(\scB)) \to \CDer(\scB)$ are the isomorphism of Lemma~\ref{lem:3.4}).
\end{lemma}

Since an integral curve of a vector field on an affine scheme
$\Spec(\scA)$ is a map from an interval to a scheme
(Definition~\ref{def:int_curve}), we need to understand maps from
manifolds to schemes.

\begin{theorem} \label{thm:3.9}
Let $\scA$ be a finitely generated and germ determined (hence
complete, cf.\ Remark~\ref{rmrk:2.10}) $\cin$-ring and
$L$ a manifold (possibly with boundary).  Let $\Pi: \cin(\R^n)\to
\scA$ be a choice of generators of $\scA$, $J= \ker\Pi$ and $Z_J$ the
set of zeros of $J$ (Definition~\ref{def:2.17}).  The map
\begin{equation}
  \begin{split}
\Psi: \Hom(\Spec (\cin(L)), \Spec(\scA))&\to \fM:= \{f:L\to \R^n\mid f
\textrm{ is } \cin \textrm{ and } f(L)\subseteq Z_J\}\\
 \Psi((\varphi, \varphi_\#))& \mapsto \varphi
\end{split}
\end{equation}  
is a bijection.   Here we identify $L $ with the set of $\R$-points
$\Hom (\cin(L), \R)$ of $\cin(L)$ and $Z_J$ with the set of
$\R$-points of $\scA$ (Lemma~\ref{lem:2.19}).
\end{theorem}

\begin{remark}
We may define the set of ``smooth'' functions on $Z_J$ to be
$\cin(Z_J):= \cin(\R^n)|_{Z_J}$.  The pair $(Z_J, \cin(Z_J))$ is then
a differential space in the sense of Sikorski \cite{Si},  and a function $f$ is in
the set of $\fM$ if and only if $f:M\to Z_J$ is smooth as a map of
differential spaces. 
\end{remark}
  
\begin{proof}[Proof of Theorem~\ref{thm:3.9}]
  Recall that by Milnor's exercise the map
  
\[
H: \cin(L, \R^n) \to \Hom(\cin(\R^n), \cin(L)),\qquad H(f) = f^*
\]
is a bijection. The image of $\fM$ is the set
\[
H(\fM) = \{\varphi:\cin(\R^n) \to \cin(L) \mid \varphi(J) = 0\}.
\]
Since $\Pi:\cin(\R^n)\to \scA$ is surjective, the  map
\[
\Pi^*: \Hom(\scA, \cin(L)) \to \Hom(\cin(\R^n), \cin(L)), \qquad \Pi^*(\mu) =
\mu\circ \Pi
\]
is injective.  Its image is again $\{\varphi:\cin(\R^n) \to \cin(L)
\mid \varphi(J) = 0\} = H(\fM)$.
Hence
\[
\Phi:\fM \to \Hom(\scA, \cin(L)),\qquad \Phi( L\xrightarrow{f} Z_J) =
(\scA \xrightarrow{\overline{f}} \cin(L)),
\]  
where $\overline{f}$ is defined by $\overline{f}(\Pi(h)) = h\circ f$
for all $h\in \cin(\R^n)$, is a bijection. Since $\scA$ and $\cin(L)$
are both finitely generated and germ determined
\[
  \Spec: \Hom (\scA, \cin(L)) \to \Hom (\Spec(\cin(L)), \Spec(\scA))
\]  
is a bijection.  Consequently for any smooth function $f:L\to \R^n$
with $f(L)\subseteq Z_J$ we get a unique map
\[
\uu{\varphi}= (\varphi, \varphi_\#):= \Spec(\overline{f}):
\Spec(\cin(L)) \to \Spec(\scA).
\]  
Moreover all maps from $\Spec(\cin(L))$ to $\Spec(\scA)$ are of the form
$\Spec(\overline{f})$ for some smooth function $f:L\to \R^n$
with $f(L)\subseteq Z_J$.

We now argue that the map $\varphi:
X_{\cin(L)} \to X_\scA$ is $f$ once we identify $L \simeq X_{\cin(L)}$
and $Z_J \simeq X_\scA$.    In other words
\begin{equation}
  \xy
(-10,6)*+{X_{\cin(L)}}="1";
(14,6)*+{X_\scA}="2";
(-10, -6)*+{L}="3";
 (14, -6)*+{Z_J}="4";
{\ar@{<-}_{} "1";"3"};
{\ar@{<-} "2";"4"};
{\ar@{->}^{\varphi} "1";"2"};
{\ar@{->}_{f} "3";"4"};
(-24,6)*+{\ev_q}="-1";
(-24,-6)*+{q}="-3";
{\ar@{|->} "-3"; "-1"};
(24,6)*+{\overline{\ev}_p}="22";
(24,-6)*+{p}="44";
{\ar@{|->} "44"; "22"};
\endxy
\end{equation}
commutes. Here and below the right vertical map is the homeomorphism
of Lemma~\ref{lem:2.19}.  Recall that by construction of
$\Spec(\overline{f})$ the map
$\varphi$ sends an $\R$-point $x: \cin(L)\to \R$ to the $\R$-point $x
\circ \overline{f}: \scA\to \R$:
\[
\varphi(x) = x\circ \overline{f}.
\] 
Since $x = \ev_q$ for some $q\in L$, and  since $\overline{f}(\Pi(h)) =
h\circ f$ and 
$\overline{\ev}_p(\Pi(h)) = h(p)$  for all $h\in \cin(\R^n)$,
\[
\varphi(\ev_q) (\Pi(h)) =  \big(\ev_q \circ \overline{f} \big)(\Pi
(h)) = \ev_q (h\circ f) = h (f(q)) = \overline{\ev}_{f(q)} h.
\]  
In other words
\[
\varphi(\ev_q)  = \overline{\ev}_{f(q)}
\]  
and we are done.

\end{proof}

\section{Integral curves of vector fields on schemes} \label{sec:int}
The goal of the section is to prove the following result.

\begin{theorem}\label{thm:int_curve} \label{thm:4.1}
Let $v$ be a vector field on an affine scheme $(X_\scA, \scO_\scA) := 
\Spec(\scA)$ where $\scA$ is a finitely generated and germ determined
(hence complete) $\cin$-ring.  Then for any point $p\in X_\scA$ there
exists a (necessarily unique) maximal integral curve $\uu{\mu}_p = (\mu_p, (\mu_p)_\#):
(K_p, \cin_{K_p})\to \Spec(\scA)$ of the vector field $v$
subject to the initial condition
$\mu_p(0) = p$ (Definition~\ref{def:int_curve}).
\end{theorem}

\begin{proof}
Since $\scA$ is finitely generated there is $n>0$ and a surjective map
\[
  \Pi: \cin(\R^n)\to \scA
\] of $\cin$-rings.  Let $J= \ker \Pi$,
\[
Z_J:= \{p \in \R^n \mid j(p) = 0\textrm{ for all }j\in J\}.
\]
(cf.\ Definition~\ref{def:2.17}).  Let $\hat{v} = \bGamma(v)\in
\CDer(\scA)$, the derivation of $\scA$ corresponding to $v$  (cf.\ \eqref{eq:3.2+}).
Under the identification $X_\scA$ with $Z_J$ (cf.\
Lemma~\ref{lem:2.19}), a point $p\in X_\scA$ corresponds to a point in
$Z_J$, which we again denote by $p$.

By Lemma~\ref{lem:4.2} there is a vector field $V $ on $\R^n$ with $\hat{v} (f+I) = V(f)+I$
for all $f\in \cin(\R^n)$, i.e., $V$ and $\hat{v}$
are $\Pi$-related.

Let $\gamma_p:I_p\to \R^n$ be
the unique maximal integral curve of the vector field $V$ 
with the initial condition $\gamma_p(0) = p$.
Let $K_p$ denote the connected component of 0 in $\gamma_p\inv (Z_J)$
and let $f= \gamma_p|_{K_p}$.  Then $f(K_p) \subseteq Z_J$ by
construction.  By Theorem~\ref{thm:3.9} we get
$\overline{f} := \overline{(\gamma_p|_{K_p})^*}:\scA\to \cin(K_p)$ and
\[
\uu{\mu}_p = (\mu_p, (\mu_p)_\#) = \Spec(\overline{f}): (K_p,
  \cin_{K_p}) \to \Spec(\scA)
\]  
with $\mu_p = \gamma_p |_{K_p}$.  Consequently $\mu_p(0) =
\gamma_p(0) = p$.  We now need a lemma:

\begin{lemma} \label{lem:4.3}
  We use the notation above. Assume $K_p\not = \{0\}$.  The derivation $\hat{v}\in
  \CDer(\scA)$  is $\overline{f}$-related to  $\frac{d}{dt}\in \CDer(\cin(K_p))$.
\end{lemma}

\begin{proof}
Since $\gamma_p$ is an integral curve of $V$
\[ 
\frac{d}{dt} (h \circ \gamma_p) = V(h)\circ \gamma_p
\]
for all $h\in \cin(\R^n)$.   Hence
\[
\frac{d}{dt} (h \circ \gamma_p|_{K_p}) = V(h)\circ \gamma_p|_{K_p}
\]
Since $f = \gamma_p|_{K_p}$ it follows that $V$ and $\frac{d}{dt}\in
\CDer(\cin(K_p))$ are $f^*$-related:
\[
  f^*\circ V = \frac{d}{dt} \circ f^*.
\]
By definition of $\overline{f}$, $f^* = \overline{f}\circ \Pi$.
Hence
\[
\frac{d}{dt} \circ \overline{f}\circ \Pi = \overline{f}\circ \Pi\circ V.
\]
Since $V$ and $\hat{v}$ are $\Pi$ related, $\overline{f}\circ \Pi\circ
V = \overline{f}\circ \hat{v} \circ \Pi$.  Therefore
\[
\frac{d}{dt} \circ \overline{f}\circ \Pi = \overline{f}\circ \hat{v} \circ \Pi.
\]
Since $\Pi$ is surjective it follows that
\[
\frac{d}{dt} \circ \overline{f}= \overline{f}\circ \hat{v}, 
\]
which is what we wanted to prove.
\end{proof}  
It follows from Lemmas~\ref{lem:3.7} and \ref{lem:4.3} that
$\frac{d}{dt}$ is $\uu{\mu}_p = \Spec(\overline{f})$-related to
$\Spec(\hat{v})$.  Since $\hat{v} = \bGamma (v)$ and since
$\bGamma\inv = \Spec$, $\Spec(\hat{v}) = v$.  We conclude that
$\uu{\mu}_p:(K_p, \cin_{K_p})\to \Spec(\scA)$ is an integral curve of
$v$ with $\mu_p(0) = p$.

It remains to prove that given any integral curve $\uu{\nu}: (L,
\cin_L) \to \Spec(\scA) $ of $v$ with $\nu(0) = p$,  the interval $L$ is a
subinterval of $K_p$, and  that
\begin{equation}
\uu{\nu}  = \uu{\mu}_p \circ \Spec(\imath^*)
\end{equation}
where $\imath: L\hookrightarrow  K_p$ is the inclusion and
$\imath^*:\cin(K_p) \to \cin(L)$ is the pullback map.  This would
prove that $\uu{\mu}_p$ is a maximal integral curve of $v$ with the
initial condition $\mu_p(0) = p$.

So let $\uu{\nu} = (\nu, \nu_\#): (L, \cin_L) \to \Spec(\scA) $ be an integral curve
of $v$ with $\nu(0) = p$.  Again we assume that $L\not = \{0\}$ (for
otherwise there is nothing to prove). Then by Lemma~\ref{lem:3.8}
$\hat{v} = \bGamma(v)$ is
$\Spec\inv (\uu{\nu})$-related to $\frac{d}{dt}$. By
Theorem~\ref{thm:3.9}, $\Spec\inv (\uu{\nu}) = \overline{\nu}$ where
$\overline{\nu}:\scA\to \cin(L)$ is the unique map of $\cin$-rings
with 
\[
\overline{\nu}\circ \Pi = \nu^*.
\]
Since $V$ is $\Pi$-related to $\hat{v}$ and since $\hat{v}$ is
$\overline{\nu}$-related to $\frac{d}{dt}$, $V$ is
$\overline{\nu}\circ \Pi = \nu^*$-related to $\frac{d}{dt}$.     That
is, $\nu:L\to \R^n$ is an integral curve of $V$ with $\nu(0) = p$, and
furthermore $\nu(L) \subseteq Z_J$.  It follows that $\nu =
\gamma_p|_L$.  Additionally, since $K_p$ is the connected component of
$\gamma_p\inv (Z_J)$ containing 0 and since $L$ is connected, contains
0 and is sent to $Z_J$ by $\nu = \gamma_p|_L$, $L$ has to be a subset
of $K_p$.  Therefore
\[
\nu = \gamma_p|_L = (\gamma_p|_{K_p} )|_L = \mu_p|_L.
\]  
Thus
\[
\nu = \mu_p\circ \imath
\]
where $\imath:L\hookrightarrow K_p$ is the inclusion.  Consequently
\[
  \overline{\nu}\circ \Pi = \nu^*  = \imath^*\circ \mu_p^* = \imath^*
  \circ \overline{\mu}_p \circ \Pi.
\]
Since $\Pi$ is surjective
\[
\overline{\nu}  = \imath^*\circ \overline{\mu}_p.
\]
Hence
\[
\uu{\nu} = \Spec(\overline{\nu})  =\Spec(\imath^*\circ \overline{\mu}_p)
= \uu{\mu}_p \circ \Spec(\imath^*).
\]  
We conclude that $\uu{\nu}: (L, \cin_L)\to \Spec(\scA)$ factors
through $\Spec(\imath^*) : (L, \cin_L) \to (K_p, \cin_{K_p})$,
i.e., through $\imath:L\hookrightarrow K_p$.  Therefore $\uu{\mu}_p$
is maximal among all integral curves of $v\in \cX(\Spec(\scA))$ with
the initial condition of passing through $p\in X_\scA$ at time 0.
This finishes the proof of Theorem~\ref{thm:int_curve}.
\end{proof}

\begin{example} \label{ex:4.3}
Denote the two coordinate functions on $\R^2$ by $x$ and $y$.
Consider the ideal $\langle y^2 \rangle \subset \cin(\R^2)$ and let
$\scA = \cin(\R^2)/\langle y^2 \rangle$ denote the quotient ring.
Since the vector field $V = \partial _x + y\partial _y$ preserves the
ideal $I = \langle y^2 \rangle$, it induces a derivation $\hat{v}$ on
$\scA$ hence a vector field $v$ on $\Spec (\scA)$.  By
Remark~\ref{rmrk:Dubuc-fg} we can identify $\Spec(\scA)$ with the
local $\cin$-ringed space $(Z,
\cE)$ where
\[
Z = \{(x,y) \in \R^2\mid y^2 = 0\} = \{(x,0) \mid x\in
\R\}.
\]
The vector field $V$ is complete and its flow $\Psi$ is given
by
\[
\Psi(x,y, t) = (x+t, ye^t)
\]
for all $(x,y, t) \in \R^3$.  Hence given $(x,0) \in Z$ the integral
curve $\gamma$ of $V$ with the initial condition $\gamma(0) = (x,0)$
is $\gamma(t) = (x+t, 0)$.   It follows that in this case $K_{(x,0)} = \R$ and
$\gamma|_{K_{(x,0)}} = \gamma$.  Note that $\gamma^* y^2 = 0$, hence
$\gamma^*:\cin(\R^2)\to \R$ induces  $\overline{f}  =
\overline{\gamma^*} : \scA = \cin(\R^2)/\langle y^2\rangle \to \R$ as we
should expect.  We then get
\[
\Spec(\overline{f}) = (\gamma, \gamma_\#): \Spec(\cin(\R)) \to \Spec(\scA)  = (Z, \cE)
\]  
where, for any open set $U\subset \R^2$,
\[
\gamma_{\#, U\cap Z} : \cE(U)  \to \cin(\gamma \inv(U))
\]
is induced by
\[
\cin(U)  /\langle y^2|_U\rangle \ni h + \langle y^2|_U\rangle \mapsto
h\circ \gamma.
\]

Similarly the vector field $\partial_x$ induces a derivation $\hat{u}$
on $\scA$ with
\[
  \hat{u} (f+\langle y^2\rangle) = \partial_x f +\langle y^2\rangle
\]
for all $f+\langle y^2\rangle \in \scA$.  Note that since
\[
\hat{u} (y+\langle y^2\rangle ) = \langle y^2\rangle \not = y +
\langle y^2\rangle  = \hat{v} (y +
\langle y^2\rangle),
\]
$\hat{u}$ and $\hat{v}$ are different derivations of $\scA$.  Hence
the corresponding vector fields $u$ and $v$ on $\Spec(\scA)$ are
different as well. None
the less the vector fields $u$ and $v$ 
have exactly the same integral curves, which  is easy to check.  It
follows that one cannot hope to recover a vector field on
$\Spec(\scA)$ from its integral curves even when the curves are
defined for all time.
\end{example}  

\begin{example} \label{ex:square}
Consider the unit square
\[
Z = \{(x,y) \in \R^2 \mid |x|, |y|\leq 1\}.
\]  
We can view $Z$ as a manifold with corners.  We can also view $Z$ as
$\Spec(\scA)$ where $\scA = \cin(\R^2)/I$, where $I$ is the vanishing
ideal of $Z$ (cf.\ Remark~\ref{rmrk:corners}).  By
Example~\ref{ex:3.2} any vector field $V$ on $\R^2$ induces a vector
  field on $Z$.  In particular we can take $V = x\partial_y -
  y\partial _x$ and the corresponding induced vector field on $Z$
  which we  denote by $v$.  The integral curve of $v$ through a point
  $p= (x,y) \in Z$ is then a restriction of the integral curve of $V$
  to the appropriate interval $K_p$.   If $x^2 + y^2 \leq 1$, then
  $K_p = \R$.  If $x^2 + y^2 = 2$ (i.e., $p = (\pm 1, \pm1)$) then
  $K_p = \{0\}$.  And if $1< x^2 + y^2 <2$, then $K_p$ is a closed interval.
\end{example}  

\section{Flows} \label{sec:flows}
The goal of this section is to finish the proof of  Theorem~\ref{thm:main}. In the
previous section (Section~\ref{sec:int}) we adressed the question of existence of unique maximal integral
curves.   Here we prove that these maximal integral curves
assemble into a flow.  But first we need to discuss the requirements
on the
affine scheme $\Spec(\scB)$ so that it can serve as the domain of
the flow. 
The intuition comes  from flows of vector
fields on manifolds (which should be the special case of flows of
vector fields on $\cin$-schemes).  Recall that given a vector field $X$ on a
manifold $M$ the flow $\Phi$ of $X$ is a $\cin$-map $\Phi:U\to M$
where $U\subset M\times \R$ is the open set given by
\[
U:= \{(p, t) \in M\times \R \mid p\in M, t\in I_p\}.
\]  
Here the open interval $I_p$ is the domain of the maximal integral
curve $\gamma_p$ of $X$ with $\gamma_p(0) = p$.   Furthermore
$\gamma_p(t) = \Phi(p, t)$ for all $(p,t)\in U$.  This is where
Definition~\ref{def:flow} comes from.  However $\Spec(\scA)$ is not a
manifold, and
Definition~\ref{def:flow}  leaves room for various choices of the
$\cin$-ring $\scB$.

To explain what these choices are let us go back to the setting of the
previous section, Section~\ref{sec:int}.   Fix a finitely generated complete
$\cin$-ring $\scA$ and a derivation $\hat{v}\in \CDer(\scA)$ that
corresponds to the vector field $v$ on $\Spec(\scA)$ (see Lemma~\ref{lem:3.4}).  Since
$\scA$ is finitely generated there is $n>0$ and a surjective map $\Pi:
\cin(\R^n) \to \scA$.  Let $J = \ker \Pi$ and let $Z_J$ be the zero
set of the ideal $J$ (Definition~\ref{def:2.17}). Then $Z_J$ is
homeomorphic to the space $X_\scA$ of points of $\scA$.  Choose $V\in
\CDer (\cin(\R^n))$ which is $\Pi$-related to $\hat{v}$ (see
Lemma~\ref{lem:4.2} ).  Let $\Psi:U\to \R^n$ denote the flow of $V$.
Then the set
\[
\cW:= \{(p,t) \in U \mid p\in Z, t\in K_p\}
\]
(where, as before, $K_p$ is the domain of the maximal integral curve
$\uu{\mu}_p= (\mu_p, {\mu_p}_\#)$ 
of $v$ with $\mu_p(0) = p$) should be the set of points $X_\scB$ of
$\scB$.  Since the integral curves of the vector field $v$ are induced by the integral
curves of the vector fields  $V$  it is reasonable to look for the
$\cin$-ring $\scB$ of the form $\cin(U)/\scI$ for some 
ideal $\scI\subset \cin(U)$ with the zero set $Z_\scI$ being $ \cW$.

To make sure that such ideal $\scI$ exists one needs to check that the
set $\cW$ is closed.  This is done in Lemma~\ref{lem:5.1} below.   The
next condition on $\scI$ comes from the way the map $\uu{\Phi}$ is going to be
constructed --- one expects the map to come in some way from the flow
$\Psi:U\to \R^n$.    If $\Psi^*J \subset \scI$ then $\Psi^*:\cin(\R^n)
\to \cin(U)$ induces
\[
\overline{\Psi^*}: \cin(\R^n)/J \to \cin(U)/\scI,\quad
\overline{\Psi^*} (f+J) = \Psi^*f +\scI.
\]  
 One may hope that $\Spec(\overline{\Psi^*})$ is the flow of $v$ and,
  in particular, that it satisfies the conditions of Definition~\ref{def:flow}.
We can try setting $\scI = \langle \Psi^*J\rangle$, the ideal
generated by the set $\Psi^*J$.   However, by Lemma~\ref{lem:5.3}
below,
\[
Z_{\Psi^*J}  = \Psi\inv (Z_J) = \{(p,t) \in U \mid  \Psi(p,t) \in Z_J\}
\]  
which may be too big: there may be points $(p,t) \in U$ with $p\not
\in Z_J$ and $\Psi(p,t) \in Z_J$.   This issue is easy to fix: let
\[
  \scI = \langle \Psi^*J \rangle + \langle \pr^*J \rangle, 
\]
where $\pr: U\to \R^n$ is defined by $\pr(p,t) = p$.  Then
\[
Z_{\langle \Psi^*J \rangle + \langle \pr^*J\rangle } = Z_{\langle
  \pr^*J\rangle } \cap Z_{\langle \Psi^*J\rangle }  = \pr\inv(Z_J)
\cap \Psi\inv (Z_J).
\]  
However if there is a point $p\in Z_J$ so that the integral curve
$\gamma_p(t) = \Psi(p,t) $ of the vector field $V$ leaves $Z_J$ and
then comes back (see Example~\ref{ex:square} above) then the set $\cW$
is strictly smaller than the intersection $\pr\inv(Z_J)
\cap \Psi\inv (Z_J)$.
Note that this problem does not occur if the vector field $v$ is
complete, i.e., all of its integral curves exist for all time (and
then any integral curve of $V$ with the initial condition in $Z_J$
stays in $Z_J$ for all time.)

To fix the last problem we can set
\begin{equation} \label{ex:5.1}
\scI := \langle \pr^*J \rangle  + \langle \Psi^*J \rangle + I_\cW,
\end{equation}  
where $I_\cW \subset \cin(U)$ is the vanishing ideal  of the closed
set $\cW$. At this point we may stop and set
\[
\scB = \cin(U)/(\langle \pr^*J \rangle  + \langle \Psi^*J \rangle +
I_\cW). 
\]  
However there is a point of view that makes this solution  not
entirely satisfactory.  Namely, recall that a vector field $X$ on a
manifold $M$ defines a map of vector bundles
\begin{equation}\label{eq:5.2}
\rho: M\times \R\to TM, \quad \rho (m, t) = t X(m).
\end{equation}
The map $\rho$  gives the trivial vector bundle $M\times \R\to M$ the
structure of a Lie algebroid over $M$.  The Lie algebroid integrates
to a Lie groupoid $\scG$. The space of objects of $\scG$ is $M$, the
space of arrows is the domain $U\subset M\times \R$ of the
flow $\Psi$ of $X$, the source and target maps of $\scG$ are the
projection on the first factor $\pr:U\to M$ and the flow $\Psi:U\to M$
respectively, the unit map $u:M\to U$ is $u(q) = (q, 0)$ and the
multiplication map is
\[
m: U\times _{\pr, M, \Psi} U \to U, \quad m((q_2, t_2), (q_1, t_1)) = (q_1, t_1 + t_2).
\]  
Given this fact, we may want to require that the vector field $v$ on
$\Spec(\scA)$ gives rise  to a groupoid internal to the category $\Aff$
of affine $\cin$-schemes, and   we would want the target map of this
groupoid to be the flow $\uu{\Phi}:\Spec(\scB) \to \Spec(\scA)$ of
$v$.

\begin{remark} \label{rmrk:cocompl}
We remind the reader that the category $\cring$ of $\cin$-rings is
complete and cocomplete \cite [Theorem~3.4.5]{Borceux}.  Hence, in
particular $\cring$ has coproducts and pushouts. Consequently the
opposite category $\op{\cring}$ has products and fiber products. Since
$\Spec$ is right adjoint, it preserves products and fiber products.
Therefore for any three $\cin$-rings $\scA, \scB, \scC$ and any two
maps $f:\scA \to \scB$, $g:\scA \to \scC$ the fiber product
$\Spec(\scB)\times_{\Spec(f), \Spec(\scA), \Spec(g)} \Spec(\scC)$
exists and equals $\Spec( \scB {\otimes_{\infty}}_ {f, \scA, g}
\scC)$, where $\scB {\otimes_{\infty}}_ {f, \scA, g}
\scC$ is the pushout of $\scB \xleftarrow{f} \scA \xrightarrow{g}
\scC$ in $\cring$.
\end{remark}

Returning to the groupoid integrating the vector field $v$, in addition to the target map $\uu{\Phi}:\Spec(\scB) \to \Spec(\scA)$
of the purported groupoid we would  want to have the source, the unit and the
multiplication maps as well.  A reasonable guess would be that these maps
should come from the structure maps of the Lie groupoid integrating
the vector field $V$ on $\R^n$.   Thus we would want
\begin{enumerate}
\item \label{5.0.i}$\pr: U\to \R^n$ to induce $\overline{\pr^*}: \scA =
  \cin(\R^n)/J\to \cin(U)/\scI  = \scB$,
\item \label{5.0.ii} $u:\R^n\to U$ to induce $\overline{u^*}: \cin(U)/\scI \to
  \cin(\R^n)/J$,
\item \label{5.0.iii}
  $m: U\times_{\pr, \R^n, \Psi} U\to U$ to induce $\overline{m^*} :
  \scB \to \scB\otimes_{\infty, \scA} \scB$, where the $\cin$-ring
  $\scB\otimes_{\infty, \scA} \scB$ is the pushout of the diagram
  $\scB\xleftarrow{\overline{\pr}} \scA \xrightarrow{\overline{\Psi}}
  \scB$ in the category $\cring$ of $\cin$-rings.
\end{enumerate}

Condition (\ref{5.0.i}) is automatic since $\pr^*J\subset \scI$.  For
(\ref{5.0.ii}) to hold we need $u^*\scI \subset J$.  Since $\pr\circ u
= \id$ and $\Psi \circ u = \id$, $u^*\pr^*J = J$ and $u^*\Psi^* J =
J$. On the other hand the inclusion $u^*I_\cW \subset J$ is far from
automatic and may fail to hold if $J$ is smaller than  the zero ideal
of the set $Z_J$.     This problem is fixable: set
\begin{equation} \label{eq:5.3}
I'_\cW := \{ g\in \cin(U) \mid g|_\cW =0 \textrm{ and } u^*g\in J\}.
\end{equation}  
Lemma~\ref{lem:5.11} shows that the zero set of the ideal $I'_\cW$ is still
$\cW$ and we then redefine $\scI$ to be
\begin{equation} \label{eq:5.44}
\scI:= \langle \pr^*J \rangle  + \langle \Psi^*J \rangle + I'_\cW.
\end{equation}

Unfortunately,  unless the vector field $v$ is complete (i.e., all of
its integral curves exist for all time), I have not been able to
construct the candidate comultiplication map $\overline{m^*}:  \scB \to
\scB\otimes_{\infty, \scA} \scB$.  This issue will be addressed
elsewhere.  The proof that for complete vector fields we do have a groupoid internal to the
category of $\cin$-schemes is given in the next section.\\[10pt]
We now proceed with the proof that vector fields on finitely generated
affine schemes have flows.  We keep the notation above.  So $\scA$ is
a finitely generated complete $\cin$-ring (since the $\cin$-ring $\scA$ is
  finitely generated, $\scA$ is  complete  if and only if $\scA$ is germ
  determined, see Remark~\ref{rmrk:2.10}),  $v$ is a
vector field on $\Spec(\scA)$,  $\hat{v}\in \CDer(\scA)$ the
corresponding derivation, $\Pi:
\cin(\R^n) \to \scA$ is a surjective map, $J = \ker \Pi$, $Z=
Z_J\subset \R^n$ is
the zero set of the ideal $J$, $V\in \CDer
(\cin(\R^n)$ is a choice of a vector field on $\R^n$ that is
$\Pi$-related on $\hat{v}$,  $\Psi:U\to \R^n$ is the flow of $V$ (so
$U$ is an open subset of $\R^n\times \R$),
$\gamma_p(t) = \Psi(p, t)$ the integral curve of $V$ with initial
condition $p$,  $\uu{\mu}_p= (\mu_p, {\mu_p}_\#)
:(K_p, \cin_{K_p}) \to \Spec(\scA) = (Z, \scO_\scA)$ the maximal integral curve of $v$
with $\mu_p(0) = p \in Z_J$, and $\scI \subset \cin(U)$ is the ideal given
by \eqref{eq:5.44}.

\begin{lemma}\label{lem:5.1}
We use the notation of the preceding paragraph.  The set
\begin{equation} \label{eq:cw}
\cW:= \{(p,t) \in U \mid p\in Z, t\in K_p\}
\end{equation}
is a closed subset of $U$.
\end{lemma}

\begin{proof}
Suppose we have a sequence $\{(x_n,t_n)\} \subset \cW$ converging to
$(x,t)\in U$.   Since $Z$ is closed, $x\in Z$.  Since the sets $K_{x_n}$ are  intervals
containing 0, they are convex. It follows that  for any $a\in [0,1]$, $(x_n,at_n)\in \cW$. Moreover  
$(x_n,at_n) \to (x, at)$. Since $(x_n,at_n)\in \cW$, $\Psi(x_n,
at_n) \in Z$.   Since $\Psi$ is continuous, $\Psi(x_n,
at_n)  \to \Psi (x,at)$.  Since $Z$ is closed, $\Psi(x, at) \in
Z$ for all $a\in [0,1]$.   Hence $(x,t)\in \cW$.
\end{proof}

\begin{remark}
As was mentioned before the set $\cW$ in Lemma~\ref{lem:5.1} may  be strictly smaller
than the set
\[
  \{(p,t) \in U \mid p\in Z, \Psi(p,t) \in Z\}
\]
since an integral curve of $V$ starting at a point $p\in Z$ can leave
$Z$ in finite time and then come back; see Example~\ref{ex:square}.
\end{remark}

\begin{lemma}\label{lem:5.11}  The zero set of the ideal $I'_\cW$
  defined by \eqref{eq:5.3} is the set $\cW$ defined by \eqref{eq:cw}.
\end{lemma}

\begin{proof}
If $(p,t) \in \cW$ then $f(p,t) = 0$ for all $f\in I'_\cW$. Hence
$\cW\subset Z_{I'_\cW}$.

Suppose next that $(p,t_0)\not \in \cW$.  We argue that there is $f\in
I'_\cW$ with $f(p,t_0)\not = 0$.  There are two cases to consider:
$t_0\not = 0$ and $t_0=0$.  Suppose first that $t_0 \not = 0$.
Since $\cW$ is closed in $U$, by a theorem of Whitney there is $g\in
\cin(U)$ so that $\{x\in U\mid g(x) =0\} = \cW$. Then $g(p,t_0)\not =
0$.  Consider $f(p,t) := t g(p,t)$. The restriction $f|_\cW$ is zero,
and $f(p,t_0) = t_0g(p, t_0) \not = 0$.  On the other hand $(u^*f) (p)
= f(p, 0) = 0$ for all $p\in \R^n$ (since $\Psi $ is a flow, the set $\R^n
\times \{0\}$ is contained in $ U$).  Hence $u^*f = 0\in J$.  We conclude that  $f\in I'_\cW$.

Next suppose $t_0 = 0$.   Since $Z_J\times \{0\} \subset \cW$, $p\not
\in Z_J$.   Therefore there is $h\in J$ with $h(p) \not = 0$.  As
before choose $g\in \cin(U)$ with $\{g= 0\} = \cW$. Now consider $f =
\left(\pr^*h \right)g$.   Then
\[
f(p, 0) = h(p) g(p, 0).
\]
Since $(p,0)\not \in \cW$, $g(p,0)\not = 0$.  Therefore $f(p, 0)\not =
0$. On the other hand
\[
u^*f  =( u^* \pr^*h ) (u^*g) = h\cdot u^*g\in J
\]
since $h\in J$.  Therefore $f\in I'_\cW$ and again we are done.
\end{proof}  
 
\begin{lemma} \label{lem:5.3}
Let $f:N\to M$ be a smooth map between manifolds, $\cI\subset \cin(M)$
an ideal and $\langle f^*\cI\rangle \subset \cin(N)$ the ideal
generated by the pullback $f^*\cI$ of $\cI$.   Then
\[
Z_{\langle f^*\cI\rangle } = f\inv (Z_\cI).
\]  
\end{lemma}  
\begin{proof}
This is a computation:
\[
  \begin{split}
x \in Z_{\langle f^*\cI\rangle }& \Leftrightarrow (f^*j)(x) = 0
\textrm{ for all }j\in \cI \Leftrightarrow j(f(x))= 0 \textrm{ for all
}j\in \cI\\ &\Leftrightarrow f(x)\in Z_\cI\Leftrightarrow  x\in f\inv(Z_\cI).
\end{split}
\]  
\end{proof}

\begin{lemma} For a manifold $M$ and any finite collection
  $\cI_1,\ldots \cI_n$ of ideals in $\cin(M)$
\[
Z_{\cI_1+\cdots +\cI_n}  =Z_{\cI_1} \cap \cdots \cap Z_{\cI_n}
\]    
\end{lemma}  
\begin{proof}
  Omitted.
\end{proof}

Now consider the ideal
\[
\scI: = \langle \Psi^*J \rangle + \langle \pr^* J \rangle + I'_\cW.
\]
where, as before, $I'_\cW$ is given by \eqref{eq:5.3}.
Then the zero set of $\scI$ is
\[
Z_\scI = Z_{\langle \Psi^*J \rangle } \cap Z_{\langle \pr^* J \rangle } \cap Z_{I'_\cW}
\]  
By Lemma~\ref{lem:5.3} $Z_{\langle \Psi^*J \rangle } = \Psi\inv (Z_J)$
and $Z_{\langle \pr^* J \rangle } =\pr\inv (Z_J)$.  Since
$\cW \subset \Psi\inv (Z_J) \cap \pr\inv (Z_J)$ and since
$Z_{I'_W} = \cW$
\[
Z_\scI  = \cW.
\]  
Let $\scB : = \cin(U)/\scI$.   By Lemma~\ref{lem:2.19} the set of
points $X_\scB$ ``is'' $Z_\scI = \cW$.  Since $\Psi^* J \subset \scI$ by
construction of $\scI$,  $\Psi^*:\cin(\R^n) \to \cin(U)$ induces
\[
\overline{\Psi^*}: \scA = \cin(\R^n)/J \to \scB = \cin(U)/\scI
\]
so that the diagram
\[
  \xy
(-14,6)*+{\cin(\R^n) }="1";
(19,6)*+{\cin(U)}="2";
(-14, -6)*+{\cin(\R^n)/J}="3";
 (19, -6)*+{\cin(U)/\scI}="4";
{\ar@{->}_{\Pi} "1";"3"};
{\ar@{->}^{\Pi'} "2";"4"};
{\ar@{->}^{\Psi^*} "1";"2"};
{\ar@{->}_{\overline{\Psi^*}} "3";"4"};
\endxy
\] commutes.  Here $\Pi'$ denotes the quotient map. 
Similarly we have
\[
  \overline{\pr^*}: \cin(\R^n)/J\to \cin(U)/\scI,
\]  
which induces
\[
\Spec( \overline{\pr^*}) :\Spec(\scB) \to \Spec(\scA),
\]  
and
\[
\overline{u^*}:\cin(U)/\scI\to \cin(\R^n)/J,
\]  
which induces
\[
\Spec( \overline{u^*}) :\Spec(\scA) \to \Spec(\scB).
\]

We now argue that
\[
\uu{\Phi}: = \Spec(\overline{\Psi^*}): \Spec(\scB)\to \Spec(\scA)
\]  
is the flow of $v$ (cf.\ Definition~\ref{def:flow}).

As we already noted, we may identify the space of points $X_\scB$ of
the $\cin$-ring
$\scB$ with
\[
  Z_\scI = \cW = \bigcup_{p\in Z_J = X_\scA} \{p\} \times  K_p.
\]
For every point $p\in Z_J$ we have a smooth map
\[
  \zeta_p:K_p \to U, \qquad \zeta_p(t)  = (p, t).
\]
By definitions of $K_p$ and $\cW$, $\zeta_p (K_p) \subseteq \cW$.
By Theorem~\ref{thm:3.9} $\zeta_p$ corresponds to
\[
\uu{\zeta}_p: (K_p, \cin_{K_p}) \to \Spec(\scB)
\]  
with $\uu{\zeta}_p = \Spec(\overline{\zeta_p^*})$ (where
$\overline{\zeta_p^*}: \scB \to \cin(K_p)$ is induced by $\zeta_p^*$).  It remain to
check that $\uu{\Phi} \circ \uu{\zeta}_p$ is the maximal integral
curve $\uu{\mu}_p$ of $v$ with the appropriate initial condition.
Since $\Psi$ is the flow of $V$,
\[
\Psi \circ \zeta_p  = \gamma_p|_{K_p}.
\]  
Then
\[
\overline{\zeta_p^* } \circ \overline{\Psi^*} \circ \Pi  =
\overline{\zeta_p^* } \circ \Pi'  \circ \Psi^* = \zeta_p^* \circ
\Psi^* = (\gamma_p|_{K_p})^*  = \overline{(\gamma_p|_{K_p})^*} \circ \Pi.
\]  
Since $\Pi$ is surjective
\[
\overline{\zeta_p^* } \circ \overline{\Psi^*} =\overline{(\gamma_p|_{K_p})^*}.
\]  
Hence
\[
\uu{\mu}_p = \Spec(\overline{(\gamma_p|_{K_p})^*}) =
\Spec(\overline{\zeta_p^* } \circ \overline{\Psi^*}) =
\Spec(\overline{\Psi^*}) \circ \Spec(\overline{\zeta_p^* } ) =
\uu{\Phi} \circ \uu{\zeta}_p.
\] 
This finishes the proof of the existence of flows and therefore of Theorem~\ref{thm:main}.

\begin{example}
Let $\scA = \cin(\R^2)/\langle y^2\rangle $ be the $\cin$-ring in
Example~\ref{ex:4.3}, $\hat{v}\in \CDer(\scA)$ the derivation induced
by $V = \partial _x + y\partial _y$ and $v\in
\cX(\Spec(\cin(\R^2)/\langle y^2\rangle))$ the corresponding vector
field.  The flow $\uu{\Phi}$ of $v$ is induced by the flow $\Psi$ of
$V$, which as we have seen in Example~\ref{ex:4.3} is the smooth map
\[
\Psi:\R^3 \to \R^2, \quad \Psi(x,y,t) = (x+t, e^t y).
\]
We have seen that all integral curves of $v$ are define for all time.
We may then take the domain of the flow $\uu{\Phi}$ to be
$\Spec(\scB)$ where
\[
  \scB = \cin(\R^3)/ (\pr^* \langle y^2\rangle +\Psi^*
  \langle y^2\rangle)
\]
Since $\pr^*y^2  = y^2$ and $\Psi^* y^2  = e^{2t}y^2$, $\scB =
\cin(\R^3)/\langle y^2\rangle$ and $X_\scB \simeq \{(x,y,t\in \R^3
\mid y=0\}$.  The flow $\uu{\Phi} =(\Phi, \Phi_\#)$ is then $\Spec(\overline{\Psi^*})$
where
\[
\overline{\Psi^*}: \cin(\R^2)/\langle y^2 \rangle \to \cin(\R^3)/\langle
y^2 \rangle, \quad \overline{\Psi^*} (f+\langle y^2\rangle ) = \Psi^*f +\langle y^2\rangle.
\]
In particular $\Phi (x,0, t) = (x+t, 0)$ for all $(x,0, t) \in X_\scB$.
\end{example}

\section{Integration of  complete vector fields to groupoids}\label{sec:6}

In this section we prove that a complete vector field on an affine
scheme $\Spec(\scA)$ (where, as before $\scA$ is finitely generated
and germ determined) integrates to a groupoid in $\Aff \subset \lcrs$.

\begin{theorem} \label{thm:groupoid}
Let $\scA$ be a finitely generated and germ determined $\cin$-ring,
$\hat{v}\in \CDer (\scA)$ a derivation which gives rise to a complete
vector field $v$ on $\Spec(\scA)$.  Then the flow $\uu{\Phi}: \scB \to
\scA$ of $v$ is the target map of a groupoid internal to $\Aff$.
\end{theorem}  

\begin{proof}
As in Section~\ref{sec:flows} we may assume that $\scA = \cin(\R^n)/J$
and that the derivation $\hat{v}$ is induced by a vector field $V$ on
$\R^n$.  We continue to use the notation of Section~\ref{sec:flows}.
Thus $\Psi:U\to \R^n$ is the flow of $V$ (with $U\subset \R^n\times
\R$ open), and we identify the set $X_\scA$ of points of $\scA$ with
the zero set $Z_J$ of the ideal $J$.  Since the vector field $v$ is
complete by assumption every integral curve of $v$ is defined for all
times.  Hence for all $p\in Z_J$ the domain $K_p$ of the integral
curve $\uu{\mu}_p$ of $v$ with initial condition $p$ is all of $\R$
and the set $\cW := \{(p,t) \in U \mid p\in Z, t\in K_p\}$ underlying the domain of the flow $\uu{\Phi}$ of $v$ is
\[
\cW= Z_J\times \R \quad  (= \pr\inv (Z_J) = \Psi\inv (Z_J)).
\]  
We set
\begin{equation}
\scI:= \langle \pr^* J\rangle + \langle \Psi^*J\rangle 
\end{equation}  
and
\[
  \scB:= \cin(U)/\scI.
\]
As before we have three induced maps $\overline{\pr^*} : \scA \to \scB$,
$\overline{\Psi^*}: \scA \to \scB$ and $\overline{u^*}: \scB \to \scA$.

We now argue that the multiplication
\[
m: U\times _{\pr, \R^n, \Psi} U\to U, \quad m ((p_2, t_2), (p_1, t_1)) = (p_1, t_1+ t_2)
\]  
induces a map
\[
\overline{m^*}: \scB \to \scB\otimes_{\infty, \scA} \scB
\]  
where $\scB \to \scB\otimes_{\infty, \scA} \scB$ is the pushout of
$\scB\xleftarrow{\overline{\pr}} \scA \xrightarrow{\overline{\Psi}}
\scB$.

For any diagram $\scC \xleftarrow{f} \scE \xrightarrow{g} \scD$ the
pushout $\scC {\otimes_\infty}_{f,\scE,g} \scD$ is the quotient
\begin{equation}
\scC {\otimes_\infty} \scD/ \langle \imath_2 (f(x))- \imath_1
(g(x))\mid x\in \scE\rangle
\end{equation}
where $\scC\xrightarrow{\imath_2} \scC\otimes_\infty \scD
\xleftarrow{g} \scD$ is the coproduct in $\cring$ and $\langle \imath_2 (f(x))- \imath_1
(g(x))\mid x\in \scE\rangle$ is the ideal in the coproduct generated
by the set $\{ \imath_2 (f(x))- \imath_1
(g(x))\mid x\in \scE\}$.
By  \cite[Proposition~2.5 on p.~26]{MR} the global sections functor
$\Gamma$ restricted to the category of manifolds
\[
\Gamma: \op{\Man} \to \cring,\qquad \Gamma( M\xrightarrow{f}N) = \cin(N)
\xrightarrow{f^*} \cin(M)
\]
sends products of manifolds to the coproducts of their rings of
functions:
\[
\cin(M\times N)  = \cin(M)\otimes_\infty \cin(N).
\]  
By \cite[Theorem~2.8 on p.~30]{MR} the  functor $\Gamma$
sends transverse fiber products in $\Man$ to pushouts in $\cring$:
\begin{equation}
\cin(M\times_{f, Z, g}N) = \cin(M){\otimes_\infty}_{f^*, \cin(Z), g^*} \cin(N).
\end{equation} 
It follows that 
\[
\cin(U\times_{\pr, \R^n, \Psi} U)  = \cin(U\times U) /\langle \imath_1
(\pr^*f) -\imath_2 (\Psi^* f) \mid f\in \cin(\R^n)\rangle,
\]  
where the maps
\[
  \imath_1, \imath_2 : \cin(U) \to \cin(U\times U)
\]
are induced by the projections $\pi_1, \pi_2: U\times U\to U$. That is,
\[
  \imath_j = {\pi_j}^*, \qquad j=1, 2.
\]
Note an unfortunate clash of notation:
\[
  \pi_1((p_2, t_2), (p_1, t_1)) = (p_2, t_2)\qquad \textrm{ and }\qquad 
  \pi_2((p_2, t_2), (p_1, t_1)) = (p_1, t_1).
\]  
Let
\[
\scK := \langle \imath_1
(\pr^*f) -\imath_2 (\Psi^* f) \mid f\in \cin(\R^n)\rangle.
\]
Then
\[
\cin(U\times_{\pr, \R^n, \Psi} U) = \cin(U\times U)/\scK.
  \]
Now consider the ideal $\scJ \subset \cin(U\times U)$ given by
\[
\scJ: =\langle \imath_1(\scI) \cup \imath_2 (\scI) \cup \scK \rangle. 
\]
The maps $\imath_1, \imath_2$ induce the maps
\[
  \overline{\imath_1},
  \overline{\imath_2} : \scA \to \cin(U\times U)/\scJ.
\]  
Moreover for any $f+J\in \scA$
\[
\overline{\imath_1} (\overline{\Psi^*}(f+J)) - \overline{\imath_2}
(\overline{\pr^*} (f+J)) = \imath_1 \circ \Psi^* (f) - \imath_2 \circ
\pr^* ( f ) + \scJ = 0 +\scJ.
\]
Hence the diagram
\[
  \xy
(-18,16)*+{\cin(\R^n)/J }="1";
(18,16)*+{\cin(U)/\scI}="2";
(-18, -6)*+{\cin(U)/\scI}="3";
 (18, -6)*+{\cin(U\times U)/\scJ}="4";
{\ar@{->}_{\overline{\pr^*}} "1";"3"};
{\ar@{->}^{\overline{\imath}_2} "2";"4"};
{\ar@{->}^{\overline{\Psi^*}} "1";"2"};
{\ar@{->}_{\overline{\imath}_1} "3";"4"};
\endxy
\]
commutes.  Note that $\scB  = \cin(U)/\scI \to \cin(U\times U)/\scJ
\leftarrow \cin(U)/\scI= \scB$ is the pushout $\scB
{\otimes_\infty}_\scA \scB$.  This is because it has the appropriate
universal property.

We now argue that
\[
  m^*: \cin(U)\to \cin(U \times _{\R^n} U)  = \cin(U\times
  U)/\scK
\]
induces
\[
  \overline{m^*} : \scB= \cin(U)/\scI \to \cin(U\times U)/ \scJ = \scB
  \otimes_{\infty, \scA} \scB
\]
that is, that $m^* (\scI) \subset \scJ$.

Recall that $((p_2, t_2), (p_1, t_1))\in U\times_{\pr,\R^n, \Psi}
U\subset U\times U$ if and only if $p_2 = \Psi(p_1, t)$. Given any
function $f\in \cin(\R^n)$ and $((p_2, t_2), (p_1, t_1))\in U\times_{\R^n}U$
\[
m^*(\pr^*f) ((p_2, t_2), (p_1, t_1)) = f (\pr(p_1, t_1+t_2) ) = f(p_1).
\]  
On the other hand
\[
\imath_2 (\pr^*f) ((p_2, t_2), (p_1, t_1)) = (\pr^*f)(\pi_2 (((p_2,
t_2), (p_1, t_1)) ) = f(\pr(p_1, t_1) ) = f(p_1).
\]  
Hence
\[
m^*(\pr^*f) = \imath_2 (\pr^*f) \qquad \textrm{ for all } f\in \cin(\R^n).
\]
Therefore
\begin{equation} \label{eq:6.4}
m^* (\pr^*J) \subset \imath_2(\pr^*J) \subset \imath_2 (\scI).
\end{equation}
Similarly given any
function $f\in \cin(\R^n)$ and $((p_2, t_2), (p_1, t_1))\in
U\times_{\R^n}U$
\[
\begin{split}
  m^*(\Psi^*f) ((p_2, t_2), (p_1, t_1)) &= f (\Psi(p_1, t_1+t_2)\\
  &= f(\Psi( \Psi(p_1, t_1), t_2) \quad \textrm{ (since $\Psi$ is a
    flow)}\\
  &= f(\Psi(p_2, t_2)) \qquad \quad \textrm{ (since $p_2 = \Psi(p_1, t_1)$ )}.
\end{split}
\]  
On the other hand
\[
\imath_1 (\Psi^*f)  ((p_2, t_2), (p_1, t_1)) = (\Psi^*f)( \pi_1
((p_2, t_2), (p_1, t_1)) )
= f (\Psi (p_2, t_2)).
\]
Hence
\[
m^*(\Psi^*f) = \imath_1 (\Psi^*f) \qquad \textrm{ for all } f\in \cin(\R^n).
\]
Therefore
\begin{equation} \label{eq:6.5}
m^* (\Psi^*J) \subset \imath_1(\Psi^*J) \subset \imath_1 (\scI).
\end{equation}
Equations \eqref{eq:6.4} and \eqref{eq:6.5} now imply that
\[
m^* \scI = m^* (\langle \pr^*J\rangle + \langle \Psi^*J\rangle)
\subset \langle (\imath_2 (\scI))\rangle + \langle (\imath_1
(\scI))\rangle \subset \scJ.
\]
This finishes the verification that $m^*: \cin(U) \to
\cin(U\times_{\R^n} U) $ induces $\overline{m^*} : \scB \to \scB
\otimes_{\infty, \scA} \scB$.

It remains to check that the maps $\overline{m^*}, \overline{\pr^*},
\overline{\Psi^*}$ and $\overline{u^*}$ define a groupoid in
$\op{\cring}$.  We omit the computation.
\end{proof}

\begin{remark}
Note that Theorem~\ref{thm:groupoid} {\em does not } claim that the
groupoid in question comes from an action of $\R = \Spec(\cin(\R))$ on
$\Spec{\scA}$. In particular it does not even claim that $\Spec(\scB) =
\Spec (\scA) \times (\R, \cin_\R)$.  Thus it still remains to be
determined if a complete vector field on $\Spec(\scA)$ gives rise to
an acton of $\R =\Spec(\cin(\R))$.
\end{remark}  

\mbox{}\\
\appendix

\section{Review of $\cin$-rings, modules and $\cin$-derivations} \label{app}

In this section we rapidly review the facts we need about
$\cin$-rings, modules over $\cin$-rings, $\cin$-derivations, local
$\cin$-rings and local $\cin$-ringed spaces. 
Our main references are \cite{MR} and \cite{Joy}.   

\begin{definition}[The category $\Euc$ of Euclidean (a.k.a. Cartesian)
  spaces]
  The objects of the category $\Euc$ are the coordinate vector spaces
  $\R^n$, $n\geq 0$,  thought of as $\cin$-manifolds.  The morphisms are
  $C^\infty$ maps.
\end{definition}

\begin{remark}
Note that all objects of $\Euc$ are finite powers of one object:
\[
\R^n = (\R^1)^n.
\]
\end{remark}

\begin{definition}\label{def:C2}
A {\sf $C^\infty$-ring} $\cC$ is a functor
\[
\cC:\Euc \to \Set
\]
from the category $\Euc$ of Euclidean spaces to the category $\Set$ of
sets that preserves finite products.
\end{definition}

\begin{remark} \label{rmrk:2.4}
Definition~\ref{def:C2} says that given a $C^\infty$ ring $\cC$ we have:
\begin{enumerate}
\item For any $n\geq 0$, $\cC(\R^n)$ is an $n$-fold product of the set
  $\scC: = \cC(\R^1)$.  This means, in particular, that 
\[
\cC(\R^n\xrightarrow{x_i} \R) =\scC\, ^n \xrightarrow{ pr_i} \scC,
\]
where $x_i$ is the $i$th coordinate function and $pr_i$ is the
projection on the $i$th factor.
\item For any smooth function $g= (g_1,\ldots, g_m):\R^n \to \R^m$ we
  have a map of sets
\[
\cC(g): \cC(\R^n) = \scC^n \to \scC^m = \cC(\R^m)
\]
with
\[
pr_j \circ \cC (g) = \cC(g_j).
\]
\item For any triple of smooth functions $\R^n\xrightarrow{g} \R^m
  \xrightarrow{f} \R$ we have three maps of sets $\cC(g)$, $\cC(f)$,
  $\cC(f\circ g) $ with
\[
\cC(f) \circ \cC(g) = \cC(f\circ g).
\]
\item $\cC(\R^0)$ is a single point set $\{*\}$.  This is because by definition
  product-preserving functors take terminal objects to terminal objects.
\end{enumerate}
\end{remark}

It follows from Remark~\ref{rmrk:2.4} that we could have equivalently
defined a $\cin$-ring as a set $\scC$ together with an infinite
collection $n$-ary operations, $n=0, 1, \ldots$, one operation
$f_\scC:\scC^n\to \scC$ for each $f\in \cin(\R^n, \R)$.  Thus we have
an equivalent definition:

\begin{definition}\label{def:cring1}
A {\sf $C^\infty$-ring} is a (nonempty) {\em set} $\scC$ together with operations 
\[
g_\scC:\scC^m\to \scC
\]
for all $m$ and all $g\in C^\infty (\R^m)$
such that for all
$n,m\geq 0$, %
 all $g\in C^\infty(\R^m)$ and all
$f_1, \ldots, f_m\in C^\infty (\R^n)$
\begin{equation} \label{eq:2.assoc}
({g\circ(f_1,\ldots, f_m)})_\scC (c_1,\ldots, c_n) =
g_\scC({(f_1)}_\scC(c_1,\ldots, c_n), \ldots, {(f_m)}_\scC(c_1,\ldots, c_n))
\end{equation}
for all $(c_1, \ldots, c_n) \in \scC^n$.
Additionally we require that for every $m >0$ and for every coordinate
function $x_j:\R^m\to \R$, $1\leq j\leq m$, 
\begin{equation}\label{eq:2.projections}
(x_j)_\scC (c_1,\ldots, c_m) = c_j.
\end{equation}
\end{definition}

It is customary not to distinguish between $\cin$-rings viewed as product
preserving $\Set$-valued functors and $\cin$-rings viewed as sets with
operations; cf.\ \cite{MR, Joy}.

\begin{remark}
A nullary operation on a $\cin$-ring $\scC$ is a map $c_\scC
:\scC^0\equiv *\to \scC$, where $c:\R^0=\{0\}\to \R^1$ is a function, i.e.,
an element $c(0)$ of $\R^1$.  Equivalently  for every real number $c\in \R$ we
have an element $c_\scC \in \scC$.  In particular we have $0_\scC$ and
$1_\scC$ in $\scC$.
\end{remark}

\begin{definition}
A {\sf map} ({\sf morphism}) from a $\cin$-ring $\cC: \Euc \to \Set$
to a $\cin$-ring $\cB:\Euc\to \Set$ is a natural transformation $\varphi:
\cC\Rightarrow \cB$.

Equivalently if we think of $\cin$-rings a sets with operations, a map
of $\cin $-rings is a map of sets that preserves all the operations.
\end{definition}   

\begin{remark} \label{rmrk:a8}
Any nonzero  $\cin$-ring $\scC$ has an underlying unital
  $\R$-algebra.  This is because to define an $\R$-algebra structure
  we need scalars, addition and multiplication (plus various
  compatibility and associativity conditions). We have already
  explained that scalars are nullary operations.

  Since $g(x,y) = xy \in
\cin(\R^2)$, we have a multiplication map $g_\scC:\scC\times \scC \to
\scC$.   We write $a_1a_2$ for $g_\scC(a_1, a_2)$.   With
this notation it is not hard to show that
\[
0_\scC a = 0_\scC\qquad \textrm{ and } 1_\scC a = a
\]  
for all $a\in \scC$.  Consequently if $0_\scC = 1_\scC$ then $\scC
=\{0_\scC\}$.  Otherwise the $\cin$ ring $\scC$ contains a copy of the
reals $\R$, which we think of as the set of scalars.
The addition is the operation defined by the function
  $f(x,y) = x+y \in \cin(\R^2)$ which we simply denote by $+:
  \scC\times \scC\to \scC$.

{\sf We will not notationally distinguish between $\cin$-rings and their
underlying $\R$-algebras.  }
\end{remark}

\begin{notation}[The category $\cring$ of $\cin$-rings] The composite
  of two morphisms of $\cin$-rings is again a morphism of
  $\cin$-rings.  Consequently $\cin$-rings form a category that we
  denote by $\cring$.
\end{notation}    

\begin{remark}
The category of $\cin$-rings is complete and co-complete for category
 theoretic reasons \cite [Theorem~3.4.5]{Borceux}. Here is, roughly,
 the argument. Product
 preserving functors commute with limits and the category $\Set$ of
 sets is complete.  Consequently $\cring$ is complete and the limits
 are computed ``object-wise.''   The co-completeness follows from the
 fact that $\cring$ is a reflective subcategory of the functor
 category $[\Euc, \Set]$ and that the functor category $[\Euc, \Set]$
 is co-complete.\\[4pt]
\end{remark}

\begin{remark}
There is an evident notion of a $\cin$-subring of a $\cin$-ring.  Such
an object can be defined as a subfunctor (if we think of $\cin$-rings
as product preserving functors).  This translates into  a subset
closed under all the operations when we think of $\cin$-rings as sets
with operations.
\end{remark}  
 The following well-known lemma is very useful for dealing with
 $\cin$-rings.
\begin{lemma}[Hadamard's lemma]\label{lem:Had}
For any smooth function $f:\R^n \to \R$  there exist smooth functions $g_1,\ldots, g_n \in C^\infty (\R^{2n}) $ so that
\begin{equation} \label{eq:Had}
f(x) - f(y) = \sum_{i=1}^n (x_i - y_i) g_i(x,y)
\end{equation}
 for any pair of points $x,y\in \R^n$.  
\end{lemma}

\begin{proof}  See \cite{MR} or \cite{Jet}.
\end{proof}

\begin{definition}
An {\em ideal} in a $\cin$-ring $\scC$ is an ideal in an underlying
$\R$-algebra $\scC$.
\end{definition}
An important consequence of Hadamard's lemma is
\begin{lemma}\label{lem:2.13}
Let $I$ be an ideal in a $\cin$-ring $\scC$. Then the quotient
$\scC/I$ $\R$-algebra is naturally a $\cin$-ring: for any $n\geq 0$,
$F\in \cin(\R^n)$
\[
f_{\scC/I}:(\scC/I)^n\to \scC/I, \qquad f_{\scC/I}(a_1+I, \ldots, a_n
+I) = f_\scC(a_1,\ldots, a_n) +I\qquad a_1, \ldots, a_n \in \scC
\]
is a well-defined operation. Moreover the quotient map $\pi: \scC\to
\scC/I$ is a map of $\cin$-rings.
 \end{lemma}  
\begin{proof}
See \cite{MR}.
\end{proof}  

\begin{remark}
For any manifold $M$ the algebra $\cin(M)$ of smooth functions is
naturally a $\cin$-ring: given $f\in \cin(\R^n)$ the corresponding $n$-ary
operation
\[
f_{\cin(M)} :\cin(M)^n\to \cin(M)
\]
is given by
\[
f_{\cin_M  } (a_1, \ldots, a_n) := f\circ (a_1,\ldots, a_n) \quad \textrm{ for all 
$a_1, \ldots, a_n\in \cin(M)$}.
\]
In particular $\cin(\R^n)$ is a $\cin$-ring for any $n\geq 0$.  In
fact $\cin(\R^n)$ is a {\sf  free} $\cin$-ring on the generators
$x_1,\ldots, x_n: \R^n\to \R$ and $\R = \cin(\R^0)$ is a free
$\cin$-ring on the empty set of generators.
\end{remark}

\begin{example} \label{ex:2.16}
Let $M$ be a manifold and $Z\subset M$ a closed subset.  Recall that a
Whitney smooth function on $Z$ is the restriction of a function $f\in
\cin(M)$ to $Z$.   We write $\cin(Z)$ for the set of all Whitney
smooth functions on $Z$.  It is a standard fact that $\cin(Z)$ is an
$\R$-algebra which is the quotient of $\cin(M)$ by the ideal of smooth
functions that vanish on $Z$.  By Lemma~\ref{lem:2.13} the
$\R$-algebra $\cin(Z)$ is a $\cin$-ring.    The fact that $\cin(Z)$ is
a $\cin$-ring can also be shown directly.
\end{example}

\begin{definition} \label{def:A.17}
An {\sf $\R$-point} of a $\cin$-ring $\scA$ is a map of $\cin$-rings $p:\scA
\to \R$.
\end{definition}

\begin{remark} \label{rmrk:Pursell} An $\R$-point of a $\cin$-ring is also a map of $\R$-algebras. The converse is true as well: a map of
  $\R$-algebras $p:\scA\to \R$ is a map of $\cin$-rings.   See
  \cite[Proposition 3.6 (b), p.~33]{MR}.  The proof relies on the so
  called  Milnor's exercise (which appears to be a result of Pursell
  --- see \cite[Chapter 8]{Pursell}): for any (second countable Hausdorff) manifold $M$
  and any $\R$-algebra map $\varphi: \cin(M)\to \R$ there is a point
  $p\in M$ so that $\varphi(f) = f(p)$ for all $f\in \cin(M)$.

In particular for any (second countable Hausdorff) manifold $M$, the
set of $\R$-points of $\cin(M)$ is in bijection with the elements of
the set $M$.  
 \end{remark}

\begin{definition}\label{def:loc_ring}
  A $\cin$-ring $\scA$ is {\sf local} if  it has exactly one $\R$-point.
 Equivalently $\scA$ is local if it has a unique maximal
ideal $\fm$ and  the quotient $\scA/\fm$ is isomorphic to $\R$ (as
$\cin$-rings).
\end{definition}

\begin{remark}
  Moerdijk and Reyes call our local $\cin$-rings {\em pointed} local
  $\cin$-rings.  
\end{remark}

\begin{remark} \label{rmrk:no-R-points}
  The zero $\cin$-ring has no $\R$-points.

There are also nonzero $\cin$-rings with no $\R$-points.  A simple
example is the quotient $\cin(\R^n)/\cin_c(\R^n)$ where $\cin_c(\R^n)$
is the ideal of the compactly supported functions.

Let $M$ be a maximal ideal of $\cin(\R^n)$ containing $\cin_c(\R^n)$.
Then $\cin(\R^n)/M$ is a field, hence $0$ is a unique maximal ideal in
$\cin(\R^n)/M$ but $\cin(\R^n)/M$ has no $\R$-points, hence is not
local in the sence of Definition~\ref{def:loc_ring}.
\end{remark}
  
\begin{lemma} \label{lem:2.20}
Let $\scC$, $\scB$ be two local $\cin$-rings and $\varphi:\scC\to
\scB$ a map of $\cin$ rings.  Then $\varphi$ maps the unique maximal
ideal of $\scC$ to the maximal ideal of $\scB$.
\end{lemma}

\begin{proof} Let $p_\scC:\scC\to \R$, $p_\scB:\scB\to \R$ denote the
  unique $\R$-points.   Then $p_\scB\circ \varphi: \scC\to \R$ is an
  $\R$-point.  Since $p_\scC$ is unique
\[
p_\scB\circ \varphi = p_\scC.
\]
\end{proof}

\begin{definition} (\cite[p.\ 44]{MR}) \label{def:germ_determined} A $\cin$-ring $\scA$ is {\sf
    germ determined} if there is an injective map from $\scA$ into a
  product of local $\cin$-rings. 
\end{definition}
It is convenient to have a condition on $\cin$-ring$\scA$ being germ
determined which is expressed in terms of its localizations
$\{\scA_x\}_{x\in X_\scA}$, see Notation~\ref{not:2.22}.

\begin{lemma} \label{lem:A.24}
Given a  $\cin$-ring $\scA$ consider  the map
\[
  \pi:\scA\to \prod_{x \in X_\scA} \scA_x, \qquad \pi(x) :=
  (\pi_x)_{x \in X_\scA}
\]
induced by the set $\{\pi_x:\scA\to \scA_x\}_{x \in X_\scA}$ of all
possible localization maps (here as before $X_\scA: = \Hom(\scA, \R)$,
the set of all $\R$-points of $\scA$). The ring $\scA$ is germ
determined if and only if the map $\pi$ above is injective.
Equivalently $\scA$ is germ determined if and only if
$\bigcap_{x\in X_\scA }I_x = 0$.  Here as before
$I_x =\ker( \pi_x:\scA \to \scA_x)$.
\end{lemma}
\begin{proof}
($\Rightarrow$) Clear since each $\scA_x$ is a local ring.\\[3pt]
($\Leftarrow$)  Let $f:\scA \to \prod_{j\in J} \scS_j$ be an injective
map into a product of local $\cin$-rings.  We want to show that $
\pi:\scA\to \prod_{x \in X_\scA} \scA_x$ is injective.   Note
that
\[
  \ker(\pi) = \bigcap _{x\in X_\scA} \ker(\pi_x)
\]  and that
similarly
\[
\ker (f) = \bigcap_{j\in J}\ker(f_j).
\]
To prove injectivity of $\pi$ it is enough to show that $\ker(\pi)
\subseteq \ker(f) = 0$.

If $\scS$ is a local $\cin$-ring and $p:\scS \to \R$
is the unique $\R$-point, then $\scS\smallsetminus \ker(p)$ is a set
of units of $\scS$.  This is because if $s\in \scS$ is not a unit,
then it belongs to a maximal ideal, and there is only one maximal
ideal, namely $\ker(p)$.

For every $j\in J$ let $p_j:\scS_j\to \R$ denote the unique
$\R$-point.  Then $p_j\circ f_j:\scA\to \R$ is an $\R$-point of
$\scA$.  If $0 \not = (p_j \circ f_j) (a)$ then $p_j (f_j(a)) \not =
0$, hence $f_j(a) $ is a unit in $\scS_j$.  By the universal property
of the localization 
$\pi_{p_j\circ f_j}:\scA \to \scA_{p_j\circ f_j}$ the map $f_j$
factors through $\pi_{p_j\circ f_j}$: there exists a unique
$\bar{f}_j: \scA_{p_j\circ f_j} \to \scS_j$ so that
\[
f_j = \bar{f}_j \circ \pi_{p_j\circ f_j}.
\]  
Hence $\ker( \pi_{p_j\circ f_j}) \subset \ker(f_j)$ for all $j$.
Therefore
\[
\ker (\pi) = \bigcap _{x\in X_\scA} \ker(\pi_x) \subseteq \bigcap _{j \in J}
\ker(\pi_{p_j\circ f_j}) \subseteq \bigcap_{j\in J} \ker (f_j)  = \ker
(f)
\]
and we are done.
\end{proof}

\begin{corollary} \label{cor:A.25}
For any $\cin$-ring $\scA$ the ring $\scB: = \scA/ \bigcap_{x\in
  X_\scA} I_x$ is germ determined.  (As before $I_x =\ker(
\pi_x:\scA \to \scA_x)$.)
\end{corollary}

\begin{proof}
By the first isomorphism theorem the map $\pi:\scA \to \prod_{x\in X_\scA}
\scA/I_x$ factors through the projection $\scA \to \scB =
\scA/\ker(\pi)$:
\[
  \xy
(-10,10)*+{\scA }="1";
(18,10)*+{\prod_{x\in X_\scA } \scA/I_x}="2";
(-10, -6)*+{\scB}="3";
{\ar@{->} ^{\pi\qquad } "1";"2"};
{\ar@{->}_{} "1";"3"};
{\ar@{->}_{\bar{\pi}} "3";"2"};
\endxy
\]  
with $\bar{\pi}$ injective.  Hence by Lemma~\ref{lem:A.24} the
$\cin$-ring $\scB$ is germ-determined.
\end{proof}

\mbox{}\\
\subsection{Modules and derivations}

\begin{definition}\label{def:mod}
   A {\sf module} over a $\cin$-ring $\scA$ is a module over the
   $\R$-algebra underlying $\scA$.
 \end{definition}

 \begin{example}
Let $E\to M$ be a (smooth) vector bundle over a manifold $M$.  Then
the set of sections $\Gamma(E)$ is a module over the $\cin$-ring
$\cin(M)$.
In particular the module $\Omega^1(M)$ of the ordinary de Rham 1-forms is a module
over the $\cin$-ring $\cin(M)$.
 \end{example} 

 \begin{example}
Given a map $\varphi:\scA\to \scB$ of $\cin$-rings, $\scB$ is a module
over $\scA$.
\end{example}
   
\begin{remark}
A module $\scM$ over a
$\cin$-ring $\scA$ is a Beck module \cite{Beck, Barr}:  just as in the
case of commutative rings --- given a $\cin$-ring $\scA$ and an
$\scA$-module $\scM$, the product  $\scA\times \scM$ together with the projection $p:\scA\times \scM\to \scA$ on the
first factor is an abelian group object in the slice category
$\cring/\scA$ of $\cin$-rings over $\scA$.
\end{remark}

\begin{definition}  \label{def:der2} Let $\scA$ be a $\cin$-ring and $\scM$ an
  $\scA$-module.  A {\sf $\cin$ derivation  of $\scA$ with values in the
  module $\scM$ }is a map 
$X:\scA\to \scM$ so
that for any $n>0$, any $f\in \cin(\R^n)$ and any $a_1,\ldots, a_n\in
\scA$
\begin{equation}\label{eq:der}
X(f_\scA(a_1,\ldots,a_n)) = \sum_{i=1}^n (\partial_i f)_\scA(a_1,\ldots, a_n)\cdot X(a_i).
 \end{equation} 
\end{definition}

\begin{notation} \label{not:der}
Given a $\cin$-ring $\scA$ and an $\scA$-module $\scM$ we denote the
set of all $\cin$-derivations with values in $\scM$ by $\CDer(\scA,
\scM)$.  If the module $\scM = \scA$ we write $\CDer(\scA)$ for $
\CDer(\scA, \scA)$.
\end{notation}
\begin{remark}
The set of derivations $\CDer(\scA, \scM)$ is an $\scA$-module.
\end{remark}

\begin{example}
Let $M$ be a smooth manifold. A $\cin$ derivation $X:\cin(M) \to
\cin(M)$ of the $\cin$-ring of smooth 
functions $\cin(M)$ with values in $\cin(M)$ is an ordinary vector
field.

The exterior derivative $d: \cin(M) \to \Omega^1(M)$ is a $\cin$
derivation of $\cin(M)$ with values in the module $\Omega^1(M)$ of the 
ordinary 1-forms.
\end{example}

Recall that given a smooth map $f:M\to N$ between two manifolds
vector fields $v\in \cX(M)$ and $w\in \cX(N)$ are {\sf $f$-related} if
$w\circ f = Tf \circ v$.  Equivalently
\[
v (f^* h) = f^* (w(h))
\]  
for all smooth functions $h\in \cin(N)$.  That is
\[
v \circ f^* = f^* \circ w.
\]  
This motivates the following
definition.

\begin{definition} \label{def:related-der}
Let $\varphi : \scA \to \scB$ be a map between two $\cin$-rings.  Two
derivations $w\in \CDer(\scA)$ and $v \in \CDer (\scB)$ are
$\varphi$-related if
\[
v \circ \varphi = \varphi \circ w.
\]  
\end{definition}  

\subsection{Finitely generated $\cin$-rings.}
\begin{remark} \label{rmrk:free}
The $\cin$-ring  $\cin(\R^n)$ is a free $\cin$-ring generated by the
standard 
coordinate functions $x_1,\ldots, x_n:\R^n \to \R$ \cite[p.~17, Proposition~1.1]{MR}. This is because the $\cin$-ring operations on
$\cin(\R^n)$ are given by composition, that is,
\[
f_{\cin(\R^n)}(a_1, \ldots, a_k)  = f\circ (a_1, \ldots, a_k)
\]  
for all $k$, all $f\in \cin(\R^k)$ and all $a_1,\ldots, a_k$.
Consequently
\[
f_{\cin(\R^n)} (x_1, \ldots, x_n)  = f\circ (x_1,\ldots, x_n) = f
\]  
for all $f\in \cin(\R^n)$.  
\end{remark}

In view of Remark~\ref{rmrk:free} the following definition makes sense.
\begin{definition}\label{def:fg}.
A $\cin$-ring $\scA$ is {\sf finitely generated} if for some
$n\geq 0$ there is a
surjective map of $\cin$-rings $\Pi:\cin(\R^n) \to \scA$.
\end{definition}


\begin{thebibliography}{WWWW}

\bibitem[AGJ]{AGJ}  J.M.\ Arms, M.\ Gotay and G.\ Jennings, Geometric
  and algebraic reduction for singular momentum maps, {\em Advances in
    Mathematics} {\bf 79}  (1990), pp.\ 43--103.

  
\bibitem[Ba]{Barr} M.\ Barr, {\em Acyclic models}, AMS, Providence,
    2002.

 \bibitem[Be]{Beck} J. M. Beck, {\em Triples, algebras and
     cohomology}, PhD dissertation, Columbia University, 1967,
   \url{http://www.tac.mta.ca/tac/reprints/articles/2/tr2abs.html}.


\bibitem[Bo]{Borceux}  F.\ Borceux, {\em Handbook of categorical
    algebra 2}, Cambridge University Press, 1994.
   
 
 \bibitem[D]{Dubuc} E.J.\ Dubuc, { $C^\infty$-schemes},  {\em Amer.\
     J.\ Math.} {\bf  103} (1981), no. 4, 683--690.

\bibitem[IZ]{IZ} P.\ Iglesias-Zemmour,
{\em Diffeology},
 {Mathematical Surveys and Monographs} {\bf 185},
{American Mathematical Society, Providence, RI},
{2013},  pp.\ xxiv+439.   
   
\bibitem[J]{Joy} D.\ Joyce, Algebraic geometry over $C^\infty$-rings,
  {\em Mem.\ AMS} {\bf 260} (2019), no.\ 1256.

\bibitem[Ka]{Ka} D.B.\ Kaledin, Normalization of a Poisson algebra is
  Poisson, {\em  Proceedings of the Steklov Institute of Mathematics}, 2009,
  Vol. 264, pp. 70--73.

\bibitem[KL]{KL} Y.\ Karshon and E.\ Lerman, Vector fields and flows
  on subcartesian spaces, {\em SIGMA. Symmetry, Integrability and
    Geometry: Methods and Applications} {\bf 19} (2023): 093.

\bibitem[Lei]{Leinster} T.\ Leinster, {\em Basic category theory},
  Cambridge University Press, 2014.


\bibitem[L1]{L-forms} E.\ Lerman, Differential forms on $\cin$-ringed
  spaces, {\em Journal of Geometry and Physics}, {\bf 196}  (2024): 105062

\bibitem[L2]{L-C} E.\ Lerman, Cartan calculus for
  $\cin$-ringed spaces, {\em  Journal of Topology and Analysis}
  (2024), pp. 1--50.


\bibitem[L3]{L_PS} E.\ Lerman,  Poisson $\cin$-rings and schemes, \href{https://arxiv.org/abs/2507.18621}{arXiv:2507.18621 [math.SG]}
  
\bibitem[MR]{MR} I.\ Moerdijk and G.E.\ Reyes, {\em Models for Smooth
    Infinitesimal Analysis}, Springer, 1991.


\bibitem[N]{Jet} J.\ Nestruev, {\em Smooth manifolds and observables},
Springer, 2003.   

\bibitem[Po]{Po} A.\ Polishchuk, Algebraic geometry of Poisson
  brackets, {\em  Journal of Mathematical Sciences} {\bf 84} no.\
  5 (1997), pp.\ 1413--1444.

\bibitem[Pu]{Pursell} L.E.\ Pursell, {\em Algebraic structures
    associated with smooth manifolds}, PhD thesis, Purdue University,
  1952.

\bibitem[Si]{Si}  R.\ Sikorski, Differential modules, {\em
    Colloq. Math.}{ \bf 24} (1971), 45--79.
    
\bibitem[SW]{SW}  J. Sniatycki and A.\ Weinstein, Reduction and
  quantization for singular momentum mappings, {\em
    Lett. Math. Phys.}{\bf 7} (1983), 155--161.

\bibitem[Spa]{Spa} K.\ Spallek, 
Differenzierbare R\"aume,
{\em Math. Ann.} {\bf 180} (1969), 269--296.

\bibitem[Spi]{Spi} D.I.\ Spivak, {\em 
Quasi-smooth derived manifolds},  
Thesis (Ph.D.)–University of California, Berkeley
2007. 131 pp.

\end{thebibliography}
\end{document}